\documentclass[a4paper,10pt]{article}
\usepackage[utf8]{inputenc}
\usepackage[english]{babel}
\usepackage{latexsym,cite,mathrsfs}
\usepackage{amsfonts}
\usepackage{amssymb, amsthm}
\usepackage{xcolor}
\usepackage{marginnote}
\usepackage{amsmath}
\usepackage{a4wide}
\usepackage[affil-it]{authblk}
\usepackage[title]{appendix}

\numberwithin{equation}{section}

\usepackage{color,enumitem,graphicx}
\usepackage[colorlinks=true,urlcolor=blue,
citecolor=red,linkcolor=blue,linktocpage,pdfpagelabels,
bookmarksnumbered,bookmarksopen]{hyperref}
\usepackage{framed}
\usepackage{comment}

\usepackage[left=2.9cm,right=2.9cm,top=2.8cm,bottom=2.8cm]{geometry}
\usepackage[hyperpageref]{backref}

\makeatletter
\providecommand\@dotsep{5}
\def\listtodoname{List of Todos}
\def\listoftodos{\@starttoc{tdo}\listtodoname}
\makeatother

\usepackage{verbatim} 
\usepackage{a4wide}
\usepackage[hyperpageref]{backref}

\def\QED{\hfill {$\square$}\goodbreak \medskip}

\let\OLDthebibliography\thebibliography
\renewcommand\thebibliography[1]{
	\OLDthebibliography{#1}
	\setlength{\parskip}{1pt}
	\setlength{\itemsep}{1pt plus 0.3ex}
}

\usepackage[hyperpageref]{backref}

\usepackage{accents}

\newtheorem{thm}{Theorem}[section]
\newtheorem{lem}{Lemma}[section]
\newtheorem{cor}{Corollary}[section]
\newtheorem{prop}{Proposition}[section]

\newtheorem{rem}{Remark}[section]

\begin{document}	
 \title
	{On  periodic and compactly supported least energy solutions  to semilinear elliptic equations with non-Lipschitz nonlinearity}
	\author{ \ Jacques Giacomoni$^{\,1}$\footnote{e-mail: {\tt jacques.giacomoni@univ-pau.fr}}, \ Yavdat Il'yasov$^{\,2,3}$\footnote{e-mail: {\tt ilyasov02@gmail.com}} \ and  Deepak Kumar$^{\,4}$\footnote{e-mail: {\tt deepak.kr0894@gmail.com}} \\
		$^1\,${\small Universit\'e  de Pau et des Pays de l'Adour, LMAP (UMR E2S-UPPA CNRS 5142) }\\ {\small Bat. IPRA, Avenue de l'Universit\'e F-64013 Pau, France}\\  
		$^2\,${\small Institute of Mathematics of UFRC RAS, 112, Chernyshevsky str., 450008, Ufa, Russia,}\\
		$^3\,${\small Instituto de Matem\'atica e Estat\'istica. Universidade Federal de Goi\'as, 74001-970, Goiania, Brazil, } \\ 
		$^4\,${\small Department of Mathematics, Indian Institute of Technology Delhi,}\\
		{\small	Hauz Khaz, New Delhi-110016, India } }

	\date{}
	\maketitle


\begin{abstract} 

We discuss the existence and non-existence of periodic in one variable and compactly supported in the other variables least energy solutions for equations with non-Lipschitz nonlinearity of the form: $-\Delta u=\lambda u^p- u^q$ in $\mathbb{R}^{N+1}$, where $ 0< q < p \leq 1$, $\lambda \in \mathbb{R}$. The approach is based on the Nehari manifold method supplemented by a one-sided constraint given through the functional  of the suitable Pohozaev identity. The limit value of the parameter $\lambda$,  where the approach is applicable, corresponds to the existence of periodic in one variable and compactly supported in the other variables least energy solutions. This value is found through the extrem values of nonlinear generalized Rayleigh quotients and the so-called curve of the critical exponents of $p,q$. Important properties of the solutions are derived, such as that they are not trivial with respect to the periodic variable and do not coincide with compactly supported solutions on the entire space $\mathbb{R}^{N+1}$.
\medskip

\medskip

\end{abstract}



%

\section{Introduction}
 We deal with the
following equation:
\begin{align}\label{1}
	&-u_{zz}-\Delta_x u  = \lambda |u|^{p-1}u- |u|^{q-1}u=:f_\lambda(u),~~ (z,x) \in D_T:=(-T,T)\times \Omega,
\end{align}
subject to the periodic boundary conditions in one variable: 
\begin{align}
				&u(-T,x)=u(T,x),~~x \in \Omega,\label{2} \\
							&u_z(-T,x)=u_z(T,x),~~x \in \Omega\label{3}
\end{align}
and zero Dirichlet  boundary condition on $\partial \Omega$:
\begin{equation}\label{4}
	u(z,x)|_{\partial \Omega}=0, ~~z \in (-T,T).
\end{equation}
Here  $ 0< q < p \leq 1$,  $0<T<+\infty$, $\Omega \subset \mathbb{R}^{N}$, $N\geq 3$  is a  bounded  domain with a
smooth boundary $\partial \Omega $, which is strictly star-shaped with
respect to a point $x_{0}\in $ $\mathbb{R}^{N}$ (which will be identified as the origin of coordinates if no confusion arises) and $\lambda$ is a real parameter. We shall use the notations  $\Delta_x = \partial^2/\partial x_1^2 +\dots+ \partial^2/\partial x_N^2$, $u_{z}=\partial u/\partial z $ and $\nabla_x:=( \partial/\partial x_1,..., \partial/\partial x_N)$.
By weak solution  of \eqref{1}-\eqref{4} we mean a critical point of the  energy functional
\begin{align*}
		\Phi_{\lambda,T}(u)=\frac{1}{2} \int_{D_T} (|\nabla_x u|^{2}+ |u_z|^{2})\,\mathrm{d}x\mathrm{d}z
		- \frac{\lambda}{p+1}\int_{D_T} | u|^{p+1} \mathrm{d}x\mathrm{d}z+\frac{1}{q+1}\int_{D_T} |u|^{q+1} \mathrm{d}x\mathrm{d}z, 
\end{align*}
for $u \in W_{per,0}^{1,2}(D_T)$, where $W_{per,0}^{1,2}(D_T)$ is the Sobolev space of functions obeyed to periodicity condition \eqref{2} and zero boundary conditions \eqref{4}, $D_{T}:=(-T,T)\times \Omega$. The meaningfulness of the periodicity condition \eqref{3}  for a weak solution $u \in W_{per,0}^{1,2}(D_T)$ of \eqref{1} is given below in Section 2.

The construction of solutions that are periodic in part of variables  are relevant in a range of applications, including the study of waves on the surface of a deep fluid, the study of convection patterns  and gravity–capillary waves in hydrodynamics, in the analysis of self-localized structures without symmetry in non-linear  media, vortex flows of an incompressible perfect fluid, and particle-shaped states in field models of elementary particles (see e.g., \cite{Alfimov, Groves, Zakharov, Zhan, Zozulya}). Problems of this type often are used as approximation models (in particular, in numerical simulations) in the study  of solutions considered on the entire space $\mathbb{R}^{N+1}$ (see e.g., \cite{BerestWei, Dyachenko, Milewski, Robinson, Robinson2}).
 
This subject is also related to the problem of finding new types of solutions, including entire solutions on the whole space which are non-radially symmetric, positive, and do not tend to 0 at infinity.  In recent decades, a number of remarkable results for elliptic problems on the existence of various new types of non-negative solutions including non-radially symmetric  in the entire space,  compactly supported and partially free boundary solutions  have been obtained (see e.g., \cite{BDI, Busca, dancer, DelPino, DiazHerIlCOUNT, Hern, Il_Perid_MatSb, Il_Perid_Izv95, Kaper1, Kaper2, Malchiodi, pucci, Serrin-Zou, vazquez}). Furthermore, it has been discovered that finding periodic solutions in part of variables can be useful for the construction of the so-called multiple ends entire and spike-layers solutions \cite{DelPino, Malchiodi} and has a relation with the construction of Delaunay's unduloids \cite{BerestWei, Delaunay, Malchiodi, Mazzeo}.


We are interested in obtaining a \textit{periodically nontrivial} solution $u$ of \eqref{1}-\eqref{4}, i.e., which satisfies $\int_{D_T}  |u_z|^{2}\,\mathrm{d}x\mathrm{d}z \neq 0$. In the case $\int_{D_T}  |u_z|^{2}\,\mathrm{d}x\mathrm{d}z =0$, we say that  $u$ is \textit{periodically trivial}. 
By a  \textit{least energy solution} (sometimes also referred to as  \textit{ground state} (cf. \cite{beres}))
of \eqref{1}-\eqref{4} we mean a weak solution $u$ of \eqref{1}-\eqref{4} which satisfies the inequality 
$
\Phi_{\lambda,T}(u)\leq \Phi_{\lambda,T}(w)
$ 
for any non-zero weak solution $w \in W_{per,0}^{1,2}(D_T)$ of \eqref{1}-\eqref{3}.

The primary aim of the present paper is to analyse the impact of the parameters $\lambda>0$ and  $T>0$ on  the existence to \eqref{1}-\eqref{4} of periodically nontrivial least energy solutions.

To construct solutions  periodic in  parts of the variables, the approach proposed by Dancer \cite{dancer} is often used. This approach is based on  the fact that  solution $u^N$ of the equation of type \eqref{1}  in the lower dimension, say in $\mathbb{R}^{N}$, can be trivially extended to the whole $\mathbb{R}^{N+1}$ by setting $u^{N+1}(z,x)=u^N(x)$, which is in fact, in our terminology periodically trivial.   In \cite{dancer}, Dancer using the Crandal-Rabinowitz theorem showed that the solutions, periodic in $z \in (-T,T)$ with some $T>0$ (as the bifurcation parameter), bifurcate from the $u^{N+1}(z,x)$. Applying this approach requires $f_\lambda$ to be differentiable and $f_\lambda'(0)<0$ (cf. \cite{dancer}). However, $f_\lambda$ in \eqref{1} is non-Lipschitz at zero.  An additional obstacle is that the bifurcation methods, as in \cite{dancer}, do not allow to make a conclusion, in a simple way,  that the obtaining bifurcation branches of  solutions consist of  least energy solutions.

In the present work, to construct solutions periodic in  parts of the variables for  \eqref{1}-\eqref{4} we use the variational approach proposed in \cite{Il_Perid_MatSb, ilDifUrav94, Il_Perid_Izv95}. This method  makes it possible to find a new type of solutions  for elliptic equations in general forms but do not require that $f \in C^1$. Moreover, it can also provide a bifurcation type result  \cite{Il_Perid_Izv95} and the investigation of the asymptotic behaviour of solutions as $T$ tends to infinity \cite{ilDifUrav94}. 



Since the non-linear term $f_\lambda(u)$ is locally non-Lipschitz,  peculiar behavior of solutions of the problem appears. In particular, it may lead to the violation of the Hopf maximum principle on the boundary of $\Omega$ and one can expect the existence of the \textit{periodic by $z$ and compactly supported in $\Omega$ solution}  that is the weak solution $u \in W_{per,0}^{1,2}(D_T)$ of  \eqref{1}-\eqref{4} satisfying supplemented boundary condition 
 \begin{equation*}
\frac{\partial u(z)}{\partial \nu }=0~~\mbox{on}~~\partial \Omega,~~ z \in (-T,T).  
\end{equation*}
Here $\nu $ denotes the unit outward normal to $\partial \Omega $. See  \cite{Diaz,pucci,Serrin-Zou, vazquez} for further details on validity of Hopf lemma. It makes sense to refer to such solutions of  \eqref{1}-\eqref{4} by a more general term as \textit{solutions with compact support in part of variables}. Thus, one can state the following problem, which is of particular interest to us:  Can elliptic equations with non-Lipschitzian nonlinearity have solutions with compact support in part of variables?

The existence of compactly supported solutions on all of variables  for the equation  
\begin{equation}  \label{Eqw}
-\Delta \psi=\lambda \psi^{p}-\psi^{q}~~~\mbox{in}~~\mathbb{R}^M,~~ M\geq 1
\tag*{$(1_{\mathbb{R}^M}^*)$}
\end{equation}
had been obtained in  celebrated results by Kaper and Kwong \cite{Kaper1}, Kaper, Kwong and Li \cite{Kaper2},  Cort\'{a}zar, M.~Elgueta and P.~Felmer \cite{CortElgFelmer-1, CortElgFelmer-2},  J. Serrin and H. \ Zou \cite{Serrin-Zou}. From these results it follows  
\begin{thm}\label{thm:KKSZ} 
Assume $0<q <p<1$, $M\geq 3$  and $\lambda=1$. Let $\psi^M$ be a non-negative $C^1$
distribution solution of \ref{Eqw} with connected support. Then the support of $\psi^M$ is a ball and supp($\psi^M$) is
radially symmetric about the center. 
Furthermore, equation \ref{Eqw} admits at most one radial symmetric
compactly supported solution  and it is a classical, i.e., $\psi^M \in C^{2}(\mathbb{R}^M)$.
\end{thm}
\begin{rem} Similar result holds for \ref{Eqw} with any $\lambda>0$ since scaling $u^{M}_{\lambda}(y):=(\sigma) ^{\frac{2}{1-q }}\cdot \psi^M(y/\sigma )$ with $\sigma_\lambda=(\lambda)^{-\frac{1-q }{2(p -q )}}$ yields a solution of  \ref{Eqw} with given $\lambda$. 
\end{rem}


This kind of results can be obtained by using the shooting methods (see \cite{Kaper1, Kaper2, CortElgFelmer-1}) and the  Alexandrov \& Serrin moving plane methods (see \cite{CortElgFelmer-2, Serrin-Zou}).
However, these approaches are difficult to apply directly to the problem \eqref{1}-\eqref{4}. Indeed,  the presence of the separate variable $z$  turns the phase space $H=\mathbb{R}$ in the shooting method into an infinite-dimensional space of functions $H=\{w(z),~z\in (-T,T)\}$, and forces the application of the moving plane method by part of the variables $x \in \mathbb{R}^N$.
Furthermore, the approaches based on shooting and moving plane methods  \cite{Kaper1, Kaper2, CortElgFelmer-1, CortElgFelmer-2, Serrin-Zou} make it difficult to obtain a least energy solution of \ref{Eqw}. We also refer to \cite{AntontsevDiaz, maagli} for the existence result of compact support solutions to singular problems by means of suitable sub and super-solution method.




In the present paper, we study the existence of compactly supported solutions  using the variational method introduced in \cite{IlEg,DiazHerIlCOUNT, DIH1, DIH}  which allows us to find compactly supported least energy solutions of \ref{Eqw} using the so-called \textit{Pohozaev functional} 
\begin{equation*}
P_{\lambda,T}(u):=\frac{1}{2^{\ast }}\int_{D_T} |\nabla_x u|^{2}\,\mathrm{d}x\mathrm{d}z+\frac{1}{2}\int_{D_T} |u_z|^{2}\,\mathrm{d}x\mathrm{d}z+\frac{1%
}{{q +1}}\int_{D_T}|u|^{{q +1}}\,\mathrm{d}x\mathrm{d}z- \frac{\lambda}{{%
p +1}}\int_{D_T}|u|^{{p +1}}\,\mathrm{d}x\mathrm{d}z,
\end{equation*}
in introducing a supplementary one-side constrain $P_{\lambda,T}(u)\leq 0$ in the Nehari manifold. Here $2^{\ast }=2N/(N-2)$, $N\geq 3$.
Thus,  we seek  the solutions of  \eqref{1}-\eqref{4} through the minimization of $\Phi_{\lambda,T}(u)$ restricted on the following Nehari manifold subset
$$
M_{\lambda,T}:=\{u \in W_{per,0}^{1,2}(D_T)\setminus 0:~\Phi_{\lambda,T}'(u)=0, P_{\lambda,T}(u)\leq 0 \}.
$$
Below we show that the solution of  \eqref{1}-\eqref{4} obtained by this approach yields a least energy solution.
It is known \cite{ilyaReil} that the applicability of the Nehari manifold method depends on  the so-called \textit{extreme values of the Nehari manifold  method}, namely, on the limit points of the set of parameters $\lambda$, $T$, $p,q$ of the problem, where the so-called 
\textit{applicability conditions of the Nehari manifold method}  
\begin{align*}
	\Phi''_{\lambda,T}(u):=\frac{d^2}{d t^2}\Phi_{\lambda,T}(tu)|_{t=1}\neq 0 \quad\mbox{and }
	P_{\lambda,T}(u)< 0,  ~~\forall u \in \mathcal{N}_\lambda 
\end{align*}
are satisfied \cite{ilyaReil, DiazHerIlCOUNT}.  
In the present work, we show that the extreme values of the Nehari manifold method for parameters $p , q$ are defined by the so-called curves of critical exponents  \cite{ilCrit, IlEg}, which allows us to introduce the following subset of exponents 
\begin{equation*}
\mathcal{E}_{s}(N):=\{(q ,p ) \in (0,1)\times (0,1):~N(1-q
)(1-p )-2(1+q )(1+p )>0\},
\end{equation*}%
delimited by the curve of the critical exponents $\{(q ,p) \in \mathcal{E}_{s}(N):~N(1-q)(1-p )-2(1+q )(1+p )=0\}$.
The main property of $\mathcal{E}_{s}(N)$ is that for star-shaped domains $%
\Omega $ in $\mathbb{R}^{N}$,  if $(q ,p )\in \mathcal{E}_{s}(N)$, any Nehari manifold minimizer $\hat{u} \in M_{\lambda,T}$ of $\Phi_{\lambda,T}$ 
is non degenerate, i.e., satisfies $\Phi_{\lambda,T}^{\prime \prime }(\hat{u})>0$. 



Furthermore, for the parameter $\lambda$, we introduce the following extreme values 
\begin{align*}
	&\lambda_{0}^T=\inf_{u\in W_{per,0}^{1,2}(D_T)\setminus 0}\lambda_{0}^T(u),\\
	&\lambda _{1P}^{D_T}=\inf_{u\in W_{per,0}^{1,2}(D_T)\setminus 0}\lambda _{1P}^{D_T}(u). 
\end{align*}
 where $\lambda_{0}^T(\cdot), \lambda _{1P}^{D_T}(\cdot)$ are nonlinear generalized Rayleigh quotients which are expressed  by exact formulas \eqref{Lambdaexpression}, \eqref{LpoT}, respectively. Furthermore, in section \ref{sec:NGRQ} we show that $0<\lambda _{1P}^{D_T}<\lambda_{0}^T<+\infty$.

Our first result is as follows
\begin{thm}
\label{thm1}  Let  $\Omega$ be a bounded strictly star-shaped
domain in $\mathbb{R}^N$ with $C^2$-manifold boundary $\partial \Omega$.
Assume $0<q <p <1$ such that $N(1-q)(1-p )-2(1+q )(1+p )>0$ and $0<T<+\infty$. Then there exists $%
\lambda^*(T)\in (\lambda _{1P}^{D_T}, \lambda^T_0)$ such that there holds:  
\begin{itemize}
\item[$(1^o)$] For  $\lambda\in [\lambda^*(T), +\infty)$, problem \eqref{1}-\eqref{4} possesses a periodic least energy solution  $u_{\lambda}^T$  such that   
\begin{equation*}
	\left\{\begin{aligned}
	&\Phi_{\lambda,T}(u_{\lambda}^T)>0, &&\mbox{if} ~~\lambda \in [\lambda^*(T), \lambda^T_0),\\
		&\Phi_{\lambda,T}(u_{\lambda}^T)=0, &&\mbox{if} ~~ \lambda =\lambda^T_0,\\
	&\Phi_{\lambda,T}(u_{\lambda}^T)<0, &&\mbox{if} ~~\lambda \in (\lambda^T_0, +\infty).
	\end{aligned}
	\right.
\end{equation*}
Moreover,   for all $\lambda\in [\lambda^*(T), +\infty)$, $u_{\lambda}^T$ satisfies  $\Phi_{\lambda,T}^{\prime\prime }(u_{\lambda}^T)>0$ and $u_{\lambda}^T \geq 0$ in $D_T$ with $u_{\lambda}^T\in C^{1,\gamma}(\overline{D_T})\cap C^{2}({D_T})$
for some $\gamma\in (0,1)$. 
\item[$(2^o)$] If $\lambda=\lambda^*(T)$, then problem \eqref{1}-\eqref{4} admits a nonnegative periodic least energy solution  $u_{\lambda^*(T)}^T$ which is compactly supported in $\Omega$. 
\item[$(3^o)$]  If  $\lambda\in (\lambda^*(T), +\infty)$, then there exists a nonnegative periodic least energy solution of \eqref{1}-\eqref{4} which is not compactly supported in $\Omega$. Moreover, if $\lambda\in [\lambda_{0}^T,+\infty)$,  every least energy periodic solution of \eqref{1}-\eqref{4} is not compactly supported in $\Omega$. 
\item[$(4^o)$] For any  $\lambda <\lambda _{1P}^{D_T}$, problem  \eqref{1}-\eqref{4} cannot have a weak solution.
\end{itemize}
\end{thm}
\begin{rem}
 Note that the assumption $0<q <p <1$ and  $N(1-q)(1-p )-2(1+q )(1+p )>0$ implies  $N\geq 3$, for $N \in \mathbb{N}$.
\end{rem}
In the following result, we derive some basic properties of periodic least energy solutions  depending on the value of the parameter $T>0$
\begin{thm}\label{thm2} Assume that the assumption of Theorem \ref{thm1} is satisfied.
 \begin{itemize}
	\item[$(1^o)$] There exist $d>0$ and $T_0> 0$ such that for any $T>T_0$ and $\lambda\in [\lambda^*(T), \lambda^*(T)+d)$ the   least energy solution  $u_{\lambda}^T$ of \eqref{1}-\eqref{4}  is periodically nontrivial, i.e., $\int_{D_T }|(u_{\lambda}^T)_z|^{2}\mathrm{d}x\mathrm{d}z\neq 0$. 
	\item[$(2^o)$]  There exists $T_1> 0$ such that for each  $T\in (0,T_1)$, any compactly supported in $\Omega$ periodic least energy  solution  $u_{\lambda^*(T)}^T$ of \eqref{1}-\eqref{4} has no compact support in $D_T$.
 \end{itemize}
\end{thm}

%


Regarding the outline of the paper: in Section \ref{sec:Prel}, we present some preliminaries including functional space setting, the Pohozaev functional and the  curve of critical exponents. In Section \ref{sec:NGRQ}, we mention nonlinear generalized Rayleigh's quotients and their extremal. Section \ref{sec:4} contains minimization arguments over the subset $M_{\lambda,T}$ of the Nehari manifold. In Sections \ref{sec:prTh1} and \ref{sec:prTh2} we prove Theorems \ref{thm1} and \ref{thm2}, respectively.  Section \ref{sec:concl} is devoted to conclusion remarks and discussion of open problems.  Afterwards, in the appendices we present some auxiliary results. In Appendix A, we mention some additional results concerning compactly supported solution to problem \ref{Eqw}. In Appendix B, we have some convergence result about the minimizers and finally in Appendix C, we prove a  Pohozaev identity for  solutions periodic in one variable. 

\section{Preliminaries}\label{sec:Prel}
Let   $0<T<+\infty$. Henceforth, we denote   $D_{T}:=(-T,T)\times \Omega$ and $D_\infty:=\mathbb{R} \times \Omega$; $\nabla_x:=(\partial/\partial x_1,..., \partial/\partial x_N)$, $u_z=\partial u/\partial z$ and  $\nabla:=(\partial/\partial z, \partial/\partial x_1,..., \partial/\partial x_N)$. We define the following spaces

\begin{itemize}
	\item $\mathcal{L}^p(-T,T)=L^p(D_{T})$, $1 < p < +\infty$
denotes the space of measurable functions $u: D_{T} \to \mathbb{R}$
with finite norm
$$
\|u\|_{\mathcal{L}^p}=\Big(\int_{D_{T}} |u(z,x)|^p\mathrm{d}x \mathrm{d}z\Big)^{1/p}.
$$
\item $W^{1,2}_0(\Omega)$ 
is the closure of $C^\infty_0(\Omega)$ in the norm 
$$
\|u\|_{1}=\left(\int_{\Omega} |\textcolor[rgb]{1,0,0}{\nabla_x} u |^2 \mathrm{d}x\right )^{1/2}.
$$
\item $W^{1,2}(D_{T})$ is 
the Sobolev space  with the norm
$$
\|u\|_{W^1_2(D_{T})}=\Big(\int_{D_{T}} (|u(z,x)|^2+|u_z|^2+|\nabla_x u|^2)\mathrm{d}x \mathrm{d}z\Big)^{1/2}.
$$
\item $W^{1,2}_{supp}(D_{T})$ is the closure of $C(-T,T; C^\infty_0(\Omega))$ in the norm $\|\cdot\|_{W^1_2(D_{T})}$.
\item $W^{1,2}_{sym}(D_\infty)$ is the closure of symmetric function 
$$
C^\infty_{0, sym}(\mathbb{R} \times \Omega):=\{\phi \in C^\infty_0(\mathbb{R}\times\Omega): ~\phi(z,x)=\phi(-z,x), z\in \mathbb{R}, x \in \Omega\}
$$ 
in the Sobolev space in the norm  $\|\cdot\|_{W^1_2(D_{\infty})}$.
\end{itemize}
It is not hard to show that any $u \in W^{1,2}_{supp}(D_{T})$ satisfies the Poincar\'e inequality 
\begin{equation}\label{POIAN}
	\|u\|_{\mathcal{L}^2(D_{T})}\leq C_{\Omega}(\|\nabla_x u\|_{\mathcal{L}^2(D_{T})})\leq C_{\Omega}(\| u_z\|_{\mathcal{L}^2(D_{T})}+\|\nabla_x u\|_{\mathcal{L}^2}(D_{T})),
\end{equation}
and any $u \in W^{1,2}_{sym}(D_\infty)$ satisfies the Poincar\'e inequality 
\begin{equation*}
	\|u\|_{\mathcal{L}^2(D_\infty)}\leq C_\infty (\|\nabla_x u\|_{\mathcal{L}^2(D_\infty)})\leq C_\infty(\| u_z\|_{\mathcal{L}^2(D_\infty)}+\|\nabla_x u\|_{\mathcal{L}^2(D_\infty)}),
\end{equation*}
where $C_\infty, C_{\Omega}<+\infty$ do not depend on $u$.
Hence $W^{1,2}_{supp}(D_{T})$ enjoys the equivalent norm
$$
\|u\|_{W^{1,2}_{supp}(D_{T})}=\| u_z\|_{\mathcal{L}^2(-T,T)}+\|\nabla_x u\|_{\mathcal{L}^2(-T,T)}.
$$
Notice that  $W^{1,2}_{supp}(D_T) \subset C(-T,T; W^{1,2}_0(\Omega))$ (see e.g., Lemma 1.2 in \cite{lionsF}). Hence for all $u \in W^{1,2}_{supp}(D_T) $, the   periodicity condition
\begin{equation*}
	u(-T,\cdot)=u(T,\cdot),
\end{equation*}
 is well defined and the set of functions 
$$
W_{per,0}^{1,2}(D_T):=\{u\in W^{1,2}_{supp}(D_T):~u(-T,\cdot)=u(T,\cdot)\}
$$  
defines  a closed subspace  in $W^{1,2}_{supp}(D_T)$. In what follows we denote 
$$
\mathcal{D}_{per}:=\{\psi \in C^\infty(-T,T; C^\infty_0(\Omega)): \psi(-T,x)=\psi(T,x),~x \in \Omega\}.
$$
For $T>0$, we define the projection $\pi^T:W^{1,2}_{sym}(D_\infty) \to W_{per,0}^{1,2}(D_T)$ as follows
$$
\pi^T(u)(z,\cdot)=u(z,\cdot),~~~z \in [-T,T], ~~u \in W^{1,2}_{sym}(D_\infty).
$$ 
We call $u \in W_{per,0}^{1,2}(D_T)$  weak solution of \eqref{1}-\eqref{4} if 
\begin{equation}\label{1Crit}
	\int_{D_T} \nabla_x u \cdot \nabla_x \phi ~ \mathrm{d}x\mathrm{d}z+\int_{D_T}  u_z \cdot  \phi_z ~ \mathrm{d}x\mathrm{d}z =  \int_{D_T}f_\lambda(u)\phi ~ \mathrm{d}x\mathrm{d}z, ~ \forall \phi \in \mathcal{D}_{per}.
\end{equation}
Observe that equality \eqref{1Crit} for $u \in W_{per,0}^{1,2}(D_T)$ implies periodicity condition \eqref{3}. Indeed, from \eqref{1Crit} it can be easily
shown that
\begin{align*}
	&0=\int_{T}^{-T}\left[(u_z(z), \phi_z(z))-(\Delta_x u(z), \phi(z))-(f_\lambda(u(z)),\phi(z)) \right]\,\mathrm{d}z\\
	&=\int_{T}^{-T}\left[-(u_{zz}(z), \phi(z))-(\Delta_x u(z), \phi(z))-(f_\lambda(u(z)),\phi(z)) \right]\,\mathrm{d}z\\
	&+(u_z(T), \phi(T))-(u_z(-T),\phi(-T)),
\end{align*}
$\forall \phi \in \mathcal{D}_{per}$. Here $(\cdot,\cdot)$ denotes the conjugation $W^{1,2}_0(\Omega)$ with its dual space.

Hence in view of \eqref{1Crit}, and since $\phi(T)=\phi(-T) \neq 0$ we obtain that the equality $u_z(T)=u_z(-T)$ holds in the distribution sense.  
\medskip

\noindent Consider the Pohozaev functional $P_{\lambda,T}(u)$, for $u \in W^{1,2}_{supp}(D_T) $. Then
\begin{equation}
	\Phi_{\lambda,T}(u)=P_{\lambda,T}(u)+\frac{1}{N}\int_{D_T} |\nabla_x u|^{2}\mathrm{d}x\mathrm{d}z,~~~\forall
	u\in W_{per,0}^{1,2}(D_T).  \label{PandE}
\end{equation}
We need the Pohozaev identity \cite{poh} for  periodic by $z$ solutions of \eqref{1}-\eqref{4} (for the proof see Appendix C)
\begin{lem}
\label{lem1} Assume that $\partial \Omega $ is a $C^{2}$-manifold and $N\geq 3$. Let  $u\in C^{1}(\overline{D_T})\cap C^{2}({D_T})$  be a  solution of \eqref{1}-\eqref{4}. Then there holds the Pohozaev identity 
\begin{equation*}
P_{\lambda,T}(u)=-\frac{1}{2N}\int_{-T}^T\int_{\partial \Omega }\left\vert \frac{%
\partial u}{\partial \nu }\right\vert ^{2}\,(x\cdot \nu (x))\mathrm{d}\sigma
(x).
\end{equation*}
Here $\mathrm{d}\sigma $ denotes the surface measure on $\partial \Omega $.
\end{lem}


\begin{prop}
\label{pradd} If $\displaystyle{\frac{d}{dt}\Phi_{\lambda,T}(tu):=\Phi_{\lambda,T}^{\prime }(tu)\leq 0}$ for $u\neq 0 $, then $\displaystyle{\frac{d}{dt}P_{\lambda}(tu):=
P^{\prime}_\lambda(tu)<0}$.
\end{prop}
\begin{proof} Note
$P_{\lambda,T }^{\prime }(tu)=\Phi_{\lambda,T}^{\prime}(tu)-\frac{2t}{N}\int_{D_T}|\nabla_x
	u|^{2}\,\mathrm{d}x\mathrm{d}z.$
Since $\Phi_{\lambda,T}^{\prime }(tu)\leq 0$, this implies $P^{\prime}_{\lambda
}(tu)\leq-(2t/N)\int_{D_T} |\nabla_x u|^{2}\,\mathrm{d}x\mathrm{d}z<0$.\QED
\end{proof}
\noindent Observe that if $\Omega $
is a star-shaped (resp. strictly star-shaped) domain with respect to the origin of coordinates of
$\mathbb{R}^{N}$, then $x\cdot \nu \geq 0$ (resp. $x\cdot \nu >0$) for all $x\in
\partial \Omega $. This and Lemma \ref{lem1} imply
\begin{cor}
\label{cor1} Let $\Omega $ be a bounded star-shaped domain in $\mathbb{R}%
^{N} $ with a $C^{2}$-manifold boundary $\partial \Omega $. Then any 
solution $u\in C^{1}(\overline{D_T})\cap C^{2}({D_T})$ of \eqref{1}-\eqref{4}
satisfies $P_{\lambda,T}(u)\leq 0$. Moreover, if $u(z, \cdot)$ 
has a compact support in $\Omega$, then $P_{\lambda,T}(u)=0$. Furthermore, in the case where $\Omega $ is strictly star-shaped, the converse is also true: if $P_{\lambda}(u)=0$ and $u\in C^{1}(\overline{D_T})\cap C^{2}({D_T})$ is a weak solution of \eqref{1}-\eqref{4}, then $u$ has a compact support in $\Omega$.
\end{cor}

%
%
%
For $u \in W^{1,2}_{per,0}(D_T)\setminus 0$, following \cite{IlEg}, we consider the system
\begin{equation}
\label{sis1} \left\{
\begin{array}{l}
 \int_{D_T}|\nabla u|^2\,\mathrm{d}x\mathrm{d}z-\lambda \int_{D_T}
|u|^{p+1}\,\mathrm{d}x\mathrm{d}z
	+\int_{D_T} |u|^{q+1} \,\mathrm{d}x\mathrm{d}z =\Phi_{\lambda,T}'(u), \\ \\
\int_{D_T}|\nabla u|^2\,\mathrm{d}x\mathrm{d}z-\lambda p  \int_{D_T} |u|^{p+1}\,\mathrm{d}x\mathrm{d}z
	+q\int_{D_T} |u|^{q+1} \,\mathrm{d}x\mathrm{d}z=\Phi_{\lambda,T}''(u),
\\ \\
\displaystyle{\frac{1}{2^*}\int_{D_T}|\nabla u|^2\,\mathrm{d}x \mathrm{d}z- \frac{\lambda}{p+1} \int_{D_T}
|u|^{p+1}\,\mathrm{d}x\mathrm{d}z
	+\frac{1}{q+1} \int_{D_T} |u|^{q+1} \,\mathrm{d}x}\mathrm{d}z=P_{\lambda,T}(u)-\frac{1}{N}\int_{D_T}|u_z|^2\,\mathrm{d}x\mathrm{d}z,
%
\end{array}
\right.
\end{equation}
where $\int_{D_T}|\nabla u|^2\,\mathrm{d}x\mathrm{d}z$, $\lambda \int_{D_T}|u|^{p+1}\,\mathrm{d}x\mathrm{d}z$, $\int_{D_T} |u|^{q+1} \,\mathrm{d}x\mathrm{d}z$ are considered as unknown values. 
A straighforward computation of the determinant leads to
$$
d:=det\begin{pmatrix}
  1 & 1 &  1 \\
  1 & p &  q \\
   \frac{1}{2^*} & \frac{1}{p+1} & \frac{1}{q+1} 
 \end{pmatrix}= (q-p)\Big(\frac{1}{2^*} -\frac{p+q}{(p+1)(q+1)}\Big)=\frac{(q-p)}{2N(p+1)(q+1)}d^*
$$
where $\displaystyle{d^*:=d^*(p,q,N)=N(1-q
)(1-p )-2(1+q )(1+p )}$. Note that $(q,p)\in \mathcal{E}_{s}(N)$ iff $d^*>0$.

\begin{lem}\label{pro}
Assume that $(q,p)\in \mathcal{E}_{s}(N)$. Let $u \in W_{per,0}^{1,2}(D_T)\setminus 0$ be such that $P_{\lambda,T }(u)\leq 0$ and $\Phi_{\lambda,T}^{\prime}(u)=0$, then $\Phi_{\lambda,T}^{\prime \prime }(u)>0$.
\end{lem}
\begin{proof} 
Let  $\Phi_{\lambda,T}^{\prime}(u)=0$. Then since $d^*> 0$, we found from \eqref{sis1} that
$$
\int_{D_T}|\nabla u|^2\,\mathrm{d}x\mathrm{d}z=\frac{2N(p+1)(q+1)}{d^*}\Phi_{\lambda,T}^{\prime \prime }(u)+ \frac{(p+1)(q+1)}{d^*}(P_{\lambda,T }(u)-\frac{1}{N}T_z(u)),
$$
where $T_z(u):=\int_{D_T}|u_z|^2\,\mathrm{d}x\mathrm{d}z$. The Poincar\'e inequality \eqref{POIAN} entails $\int_{D_T}|\nabla u|^2\,\mathrm{d}x\mathrm{d}z>0$ for $u\neq 0$. Hence the inequalities $P_{\lambda,T }(u)-\frac{1}{N}T_z(u) \leq 0$ and $d^*> 0$ imply that $\Phi_{\lambda,T}^{\prime \prime }(u)>0$.
\QED
\end{proof}

\section{Nonlinear generalized Rayleigh's quotients}\label{sec:NGRQ}
Our approach will be based on using a \textit{nonlinear generalized Rayleigh
quotient method} (see \cite{ilyaReil}). For $u\neq 0$,  we first introduce the so-called zero energy level Rayleigh's quotient \cite{DiazHerIlCOUNT}
\begin{equation*}
\mathcal{R}^0(u)=\frac{\frac{1}{2}\int_{D_T }|\nabla u|^{2}\mathrm{d}x\mathrm{d}z+\frac{1}{{q +1}%
}\int_{D_T }|u|^{{q +1}}\mathrm{d}x\mathrm{d}z}{\frac{1}{{p +1}}\int_{D_T }|u|^{{%
p +1}}\mathrm{d}x\mathrm{d}z}  
\end{equation*}%
and consider 
\begin{equation*}
\mathcal{R}^0(tu)=\frac{\frac{t^{1-p }}{2}\int_{D_T }|\nabla
u|^{2}\mathrm{d}x\mathrm{d}z+\frac{t^{q -p }}{{q +1}}\int_{D_T }|u|^{{q +1}%
}\mathrm{d}x\mathrm{d}z}{\frac{1}{{p +1}}\int_{D_T }|u|^{{p +1}}\mathrm{d}x\mathrm{d}z},~~t>0,~~u\neq 0.
\end{equation*}%
Notice that for any $u\neq 0$ and $\lambda \in \mathbb{R}$, 
\begin{equation*}
\text{if}~\mathcal{R}^0(u)=\lambda ,~~\text{then}%
~~\Phi_{\lambda,T}(u)=0.  
\end{equation*}%
%
%
%
%
%
%
%
%
%
%
%
%
%
%
It is easy to see that $\frac{d}{dt} \mathcal{R}^0(tu)=0$ if and only if 
\begin{equation*}
(1-p )\frac{t^{-p }}{2}\int_{D_T }|\nabla u|^{2}\mathrm{d}x\mathrm{d}z+(q -p
)\frac{t^{q -p -1}}{{q +1}}\int_{D_T }|u|^{{q +1}}\mathrm{d}x\mathrm{d}z=0
\end{equation*}%
and that the only solution to this equation is 
\begin{equation*}
 t_{0}(u)=\left(\frac{2(p-q)}{(1-p)(q+1)}\frac{\int_{D_T }|u|^{{q +1}}\mathrm{d}x\mathrm{d}z}{\int_{D_T }|\nabla u|^{2}\mathrm{d}x\mathrm{d}z} \right)^{\frac{1}{1-q}}. 
\end{equation*}
Let us emphasize that $t_{0}(u)$ is a value where the function $\mathcal{R}^0(tu)$
attains its global minimum. Substituting $t_{0}(u)$ into $\mathcal{R}^0(t)$ we
obtain the  \textit{nonlinear generalized Rayleigh quotient}: 
\begin{equation}\label{Lambdaexpression}
\lambda _{0}^T(u)=\mathcal{R}^0(t_{0}(u)u)=c_{0}^{q ,p }\frac{\left(\int_{D_T}|\nabla u|^{{2}}\mathrm{d}x\mathrm{d}z\right)^{\frac{p-q}{1-q}}\left(\int_{D_T }|u|^{{q +1}}\mathrm{d}x\mathrm{d}z\right)^{\frac{1-p}{1-q}}}{\int_{D_T}|u|^{{p +1}}\mathrm{d}x\mathrm{d}z},
\end{equation}%
where
\begin{equation*}
c_{0}^{q ,p }=\frac{(1-p)(q+1)}{(1-q)(p+1)}\left(\frac{(1-q)(p+1)}{2(q-p)}\right)^\frac{p-q}{1-q}.
\end{equation*}%
Furthermore, it is not difficult to check that 
\begin{eqnarray}\label{T_1}
\Phi_{\lambda_0^T(u),T}(t_0(u)u)=0, \quad\Phi_{\lambda_0^T(u),T}'(t_0(u)u)=0,
\end{eqnarray}
and to prove that
the map $\lambda^T (\cdot ):W^{1,2}_{supp}(D_T)\setminus
0\rightarrow \mathbb{R}$ is a $C^{1}$-functional. Consider 
\begin{equation}
\lambda _{0}^T=\inf_{u\in W_{per,0}^{1,2}(D_T)\setminus 0}\lambda _{0}^T(u)
\label{P3}
\end{equation}%
and
\begin{equation}
\lambda _{0}^\infty=\inf_{u\in W^{1,2}_{sym}(D_\infty)\setminus 0}\lambda _{0}^\infty(u).
\label{P3I}
\end{equation}%
Using Sobolev's, Poincar\'e's  and Holder's inequalities  it can be
shown that 
\begin{equation*}
0<\lambda _{0}^T<+\infty \mbox{ for all }T \in (0,+\infty) \mbox{ and }T=+\infty.
\end{equation*}%
It is easy to show the following
\begin{prop}
\label{propL0}
For $T>0$ ($T=\infty$)
\begin{description}
\item[(i)] If $\lambda <\lambda _{0}^T$, then $\Phi_{\lambda,T}(u)>0$ for any $u
\neq 0$,
\item[(ii)] If $\lambda >\lambda _{0}^T$, then there is $u\in W_{per,0}^{1,2}(D_T) \setminus 0$ ($u\in W^{1,2}_{sym}(D_\infty)\setminus 0$) such that $\Phi_{\lambda,T}(u)<0$.
\end{description}
\end{prop}
{Furthermore, we are able to prove}
\begin{lem}\label{prop:EXlam0}
There exists a  minimizer $\hat{u}_0^T \in  W_{per,0}^{1,2}(D_T)\setminus 0$ ($\hat{u}_0^\infty \in  W^{1,2}_{sym}(D_\infty)\setminus 0$) of \eqref{P3} (respectively \eqref{P3I}). Moreover, $D\Phi_{\lambda _{0}^T,T}(\hat{u}_0^T)=0$ and $\Phi_{\lambda _{0}^T,T}%
(\hat{u}_0^T)=0$, for all $T>0$ (respectively $T=+\infty$).
\end{lem}
\begin{proof} Let $T>0$ ($T=\infty$) and let $(u_n)_{n=1}^\infty$ be a minimizing sequence of \eqref{P3} (\eqref{P3I}). Due to the homogeneity of $\lambda _{0}^T(u)$ ($\lambda _{0}^\infty(u))$, one can find constants $a_n$ such that  $\lambda _{0}^T(a_nu_n)=\lambda _{0}^T(u_n)$ ($\lambda _{0}^\infty(a_nu_n)=\lambda _{0}^\infty(u_n)$) and $\|v_n\|_1=1$, where $v_n=a_nu_n$ for $n\in\mathbb{N}$. Hence $(v_n)$ is a minimizing sequence of \eqref{P3} (\eqref{P3I}) which is bounded in $W_{per,0}^{1,2}(D_T)$ ($W^{1,2}_{sym}(D_\infty)$). Thus there exists a subsequence,  again denoted by $(v_n)$, such that
 $v_n\rightharpoonup v_0$ weakly in $W^{1,2}(D_T)$ (respectively $W^{1,2}_{sym}(D_\infty)$) and strongly $%
v_n\rightarrow v_0$ in $\mathcal{L}^\gamma(-T,T)$, $0<T\leq +\infty$, for $1<\gamma<2^{\ast }$ and for some $v_0 \in W_{per,0}^{1,2}(D_T)$ (respectively $v_0 \in W^{1,2}_{sym}(D_\infty)$). Note that $v_0 \neq 0$. Indeed, in the opposite case we have
\begin{align*}
	\lambda _{0}^T(v_n)=c_{0}^{q ,p }\frac{\left(\int_{D_T }|v_n|^{{q +1}}\mathrm{d}x\mathrm{d}z\right)^{\frac{1-p}{1-q}}}{\int_{D_T}|v_n|^{{p +1}}\mathrm{d}x\mathrm{d}z}\geq C \left(\int_{D_T}|v_n|^{{q +1}}\mathrm{d}x\mathrm{d}z\right)^{-\frac{(2^*-2)(p-q)}{(2^*-(1+q))(1-q)}}\to +\infty,
\end{align*}
since by Sobolev's, Poincar\'e's and Holder's inequalities,
$$
\int_{D_T }|v_n|^{{p +1}}\mathrm{d}x\mathrm{d}z\leq \left(\int_{D_T}|v_n|^{{q +1}}\mathrm{d}x\mathrm{d}z\right)^{\frac{\gamma}{1+q}}\left(\int_{D_T}|v_n|^{2^*}\mathrm{d}x\mathrm{d}z\right)^{\frac{1+p-\gamma}{2^*}}\leq C\left(\int_{D_T}|v_n|^{{q +1}}\mathrm{d}x\mathrm{d}z\right)^{\frac{\gamma}{1+q}}
$$
where $\gamma=\frac{(2^*-(1+p))(1+q)}{2^*-(1+q)}$ and $0<C<+\infty$ does not depend on $n\in\mathbb{N}$. The same holds for $(v_n) \subset  W^{1,2}_{sym}(D_\infty)$. Observe that $\lambda _{0}^T(\cdot)$ is weakly lower semicontinuous and bounded below functional on $ W_{per,0}^{1,2}(D_T)$. This easily yields that  $\hat{u}_0^T:=v_0 \in  W_{per,0}^{1,2}(D_T)\setminus 0$ is a  minimizer of $\lambda _{0}^T(u)$. Then 
\begin{equation*}
D\lambda _{0}^T(\hat{u}_0^T)(\phi )=D\mathcal{R}^0(t \hat{u}_0^T)|_{t=t_{0}(\hat{u}_0^T)}(\phi )+\frac{\partial }{%
\partial t}\mathcal{R}^0(t_{0}(\hat{u}_0^T)\hat{u}_0^T)(Dt_{0}(\hat{u}_0^T)(\phi ))=0,~~\forall \phi \in
\mathcal{D}_{per}.
\end{equation*}%
Since $\partial \mathcal{R}^0(t_{0}(\hat{u}_0^T)\hat{u}_0^T)/\partial t=0$, this implies
\begin{equation*}
D\mathcal{R}^0(t_{0}(\hat{u}_0^T)\hat{u}_0^T)(\phi )=t_{0}(\hat{u}_0^T)\cdot
D\mathcal{R}^0(u)|_{u=t_{0}(\hat{u}_0^T)\hat{u}_0^T}(\phi )=0,~~~\forall \phi \in \mathcal{D}_{per}.
\end{equation*}%
Now taking into account that the equality $\lambda _{0}^T=\lambda _{0}^T(\hat{u}_0^T)$
implies $\Phi_{\lambda _{0}^T,T}(\hat{u}_0^T)=0,$ we obtain 
\begin{equation*}
0=D\mathcal{R}^0(u)|_{u=t_{0}(\hat{u}_0^T)\hat{u}_0^T}=\frac{1}{\int_{D_T}|u|^{{p +1}}\mathrm{d}x\mathrm{d}z}%
\cdot D\Phi_{\lambda _{0}^T,T}(u)|_{u=t_{0}(\hat{u}_0^T)\hat{u}_0^T},
\end{equation*}%
which yields that $D\Phi_{\lambda _{0}^T,T}(\hat{u}_0^T)=0$.\QED
\end{proof}
\noindent From the proof of Lemma \ref{prop:EXlam0} one gets also the following
\begin{cor}\label{cor:EXlam0}
 Let $T>0$ ($T=\infty$) and let $(u_n)_{n=1}^\infty$ be a minimizing sequence of \eqref{P3} (\eqref{P3I}) such that $(\|u_n\|_1)_{n=1}^\infty$ is bounded.  Then up to a subsequence, there exists a non-zero strong limit point of 	$(u_n)_{n=1}^\infty$ in  $W_{per,0}^{1,2}(D_T)$ (respectively $W^{1,2}_{sym}(D_\infty)$).
\end{cor}
\noindent We shall also need the following Rayleigh's quotients: 
\begin{align*}
& \mathcal{R}^P(u)=\frac{\frac{1}{2^{\ast }}\int_{D_T}|\nabla_x u|^{2}\,\mathrm{d}x\mathrm{d}z +\frac{1}{2}\int_{D_T}|u_z|^{2}\,\mathrm{d}x\mathrm{d}z +\frac{1}{{q +1}}\int_{D_T}|u|^{{q +1}}\mathrm{d}x\mathrm{d}z}{\frac{1}{{p +1}%
}\int_{D_T}|u|^{{p +1}}\mathrm{d}x\mathrm{d}z},  
 \\
& \mathcal{R}^1(u)=\frac{\int_{D_T}|\nabla_x u|^{2}\,\mathrm{d}x\mathrm{d}z+\int_{D_T}|u_z|^{2}\,\mathrm{d}x\mathrm{d}z + \int_{D_T}|u|^{{q +1}}\mathrm{d}x\mathrm{d}z}{\int_{D_T}|u|^{{p +1}}\mathrm{d}x\mathrm{d}z}, ~~u\neq 0.
\end{align*}%
Notice that for any $u\neq 0$ and $\lambda \in \mathbb{R}$, 
\begin{equation}
\mathcal{R}^P(u)=\lambda \Leftrightarrow P_{\lambda,T}(u)=0~~\mbox{and}%
~~\mathcal{R}^1(u)=\lambda \Leftrightarrow \Phi_{\lambda,T}^{\prime }(u)=0.  \label{RPR1}
\end{equation}%
Arguing as above for $\mathcal{R}^0(tu),$ it can be shown that
each of  functions $\mathcal{R}^P(tu)$, $\mathcal{R}^1(tu)$ attains its global minimum at some point, $t_{P}(u)$
and $t_{1}(u)$, respectively.

Moreover, it is easily seen that the following
equation 
\begin{equation*}
 \mathcal{R}^P(tu)=\mathcal{R}^1(tu),~~~t>0,  
\end{equation*}%
has a unique solution 
\begin{equation*}
t_{1P}(u)=\left(\frac{(p-q)}{(q+1)}\frac{\int_{D_T }|u|^{{q +1}}\mathrm{d}x\mathrm{d}z}{\frac{(2^*-p-1)}{2^*}\int_{D_T }|\nabla_x u|^{2}\mathrm{d}x\mathrm{d}z+\frac{1-p}{2}\int_{D_T }| u_z|^{2}\mathrm{d}x\mathrm{d}z} \right)^{\frac{1}{1-q}}. 
\end{equation*}%
Thus, we are able to introduce the following \textit{nonlinear generalized Rayleigh quotient} 
\begin{equation}\label{LpoT}
	\begin{aligned}
		\lambda _{1P}^{D_T}(u):=& \mathcal{R}^P(t_{1P}(u)u)=\mathcal{R}^1(t_{1P}(u)u)
		\\
		=&c_{1P}^{q ,p }\frac{\left(\frac{2^*-1-q}{2^*}\int_{D_T }|\nabla_x u|^{2}\mathrm{d}x\mathrm{d}z+\frac{1-q}{2}\int_{D_T }| u_z|^{2}\mathrm{d}x\mathrm{d}z\right)\left(\int_{D_T }|u|^{{q +1}}\mathrm{d}x\mathrm{d}z\right)^{\frac{1-p}{1-q}}}{\left(\int_{D_T}|u|^{p+1}\mathrm{d}x\mathrm{d}z\right) \left(\frac{(2^*-p-1)}{2^*}\int_{D_T }|\nabla_x u|^{2}\mathrm{d}x\mathrm{d}z+\frac{1-p}{2}\int_{D_T }| u_z|^{2}\mathrm{d}x\mathrm{d}z\right)^{\frac{1-p}{1-q}}},
		\end{aligned}
\end{equation}
where 
\begin{equation*}
c_{1P}^{q ,p }=\frac{p+1}{p-q}\left(\frac{p-q}{1+q}\right)^{\frac{1-p}{1-q}}. 
\end{equation*}%
%
Notice that 
\begin{equation*}
P_{\lambda _{1P}^{D_T}(u)}(t_{1P}(u)u)=0\quad\mbox{and } \Phi_{\lambda
_{1P}(u),T}^{\prime }(t_{1P}(u)u)=0,~~\forall u\neq 0.  
\end{equation*}%
Recall that by Corollary \ref{cor1}, if $u \in C^{1}(\overline{D_T})\cap C^{2}({D_T})$ is a compactly supported solution of \eqref{1}, then  
$$
P_{\lambda,T}(u)=0,~~~\Phi_{\lambda,T}^{\prime }(u)=0
$$
 which implies that $t_{1P}(u)=1$ and $\lambda=\lambda _{1P}^{D_T}(u)$.
Consider 
\begin{equation}
\lambda _{1P}^{D_T}=\inf_{u\in W_{per,0}^{1,2}(D_T)\setminus 0}\lambda _{1P}^{D_T}(u).  \label{PPoh2}
\end{equation}%
Using Sobolev's and H\"{o}lder's inequalities it can be shown (see, e.g., 
\cite{IlEg}) that 
\begin{equation*}
0<\lambda _{1P}^{D_T}<+\infty .
\end{equation*}
\noindent Similar to \eqref{PPoh2} we introduce 
\begin{equation*}
	\lambda_{1P}^{\infty}=\inf_{u\in W_{per,0}^{1,2}(D_{\infty})\setminus 0}\lambda_{1P}^{D_\infty}(u),
\end{equation*}%
where
\begin{equation*}
	\begin{aligned}
		\lambda_{1P}^{D_\infty}(u):=	
		c_{1P}^{q ,p } \frac{\left(\frac{2^*-1-q}{2^*}\int_{D_\infty}|\nabla_x u|^{2}\mathrm{d}x\mathrm{d}z+\frac{1-q}{2}\int_{D_\infty}| u_z|^{2}\mathrm{d}x\mathrm{d}z\right)\left(\int_{D_\infty}|u|^{{q +1}}\mathrm{d}x\mathrm{d}z\right)^{\frac{1-p}{1-q}}}{\left(\int_{D_\infty}|u|^{p+1}\mathrm{d}x\mathrm{d}z\right) \left(\frac{(2^*-p-1)}{2^*}\int_{D_\infty }|\nabla_x u|^{2}\mathrm{d}x\mathrm{d}z+\frac{1-p}{2}\int_{D_\infty}| u_z|^{2}\mathrm{d}x\mathrm{d}z\right)^{\frac{1-p}{1-q}}}.
	\end{aligned}
\end{equation*}
As above, we  have
\begin{equation*}
	0<\lambda_{1P}^{\infty}<+\infty.
\end{equation*}

\begin{lem}
\label{propEst} For any $u\in W_{per,0}^{1,2}(D_T)\setminus 0$,

\begin{description}
\item[(i)] $ \mathcal{R}^P(tu)>\mathcal{R}^1(tu)$ iff $t\in (0,t_{1P}(u))$ and $%
 \mathcal{R}^P(tu)<\mathcal{R}^1(tu)$ iff $t\in (t_{1P}(u),+\infty )$;
\item[(ii)] $t_{1}(u)<t_{1P}(u)<t_{P}(u)$;
 \item[(iii)] $t_1(u)<t_{1P}(u)<t_0(u)$.
\end{description}
\end{lem}

\begin{proof} Observe that $ \mathcal{R}^P(tu)/\mathcal{R}^1(tu)\rightarrow
\frac{p +1}{{q +1}}>1$ as $t\rightarrow 0$. Hence, from the
uniqueness of $t_{1P}(u)$ we obtain \textbf{(i)}.

\noindent By \eqref{RPR1} we have $\Phi_{\lambda _{1P}^{D_T}(u), T}^{\prime }(u)=0$.
Therefore Proposition \ref{pradd} implies $\frac{d}{dt}P_{\lambda
_{1P}(u)}(t_{1P}(u)u)<0$. Hence and since
\begin{equation*}
\frac{d}{dt} \mathcal{R}^P(tu)|_{t=t_{1P}(u)}=\frac{p +1}{\int_{D_T} |tu|^{{p +1}%
}\mathrm{d}x\mathrm{d}z}\cdot \frac{d}{dt}P_{\lambda _{1P}^{D_T}(u), T}(tu)|_{t=t_{1P}(u)},
\end{equation*}%
we conclude that $\frac{d}{dt} \mathcal{R}^P(tu)|_{t=t_{1P}(u)}<0$. Now taking
into account that $t_{P}(u)$ is a point of global minimum of $ \mathcal{R}^P(tu)$
we obtain that $t_{1P}(u)<t_{P}(u)$. To prove $t_{1}(u)<t_{1P}(u)$, first
observe that
\begin{equation*}
\frac{d}{dt}\mathcal{R}^1(tu)|_{t=t_{1P}(u)}=\frac{t}{%
\int_{D_T}|tu|^{{p +1}}\mathrm{d}x\mathrm{d}z}\cdot \Phi_{\lambda _{1P}^{D_T}(u), T}^{\prime \prime
}(tu)|_{t=t_{1P}(u)},
\end{equation*}%
and  by Lemma \ref{pro} the equalities $\Phi_{\lambda _{1P}^{D_T}(u), T}^{\prime
}(t_{1P}(u)u)=0$, $P_{\lambda _{1P}^{D_T}(u), T}(t_{1P}(u)u)=0$ imply that \, $\Phi_{\lambda _{1P}^{D_T}(u), T}^{\prime \prime }(t_{1P}(u)u)>0$. Thus $\frac{d}{dt}%
\mathcal{R}^1(t_{1P}(u)u)>0$ and the proof of \textbf{(ii)} follows.\par 
\noindent Note that
$$
(\mathcal{R}^0)'(tu)=\frac{(p+1)}{t}\left(\mathcal{R}^1(tu)-\mathcal{R}^0(tu)\right).
$$
Hence,
\begin{equation*}
	(\mathcal{R}^0)'(tu)=0~\Leftrightarrow~\mathcal{R}^1(tu)=\mathcal{R}^0(tu)
\end{equation*}
is satisfied by $t=t_0(u)$.
Furthermore, if $(\mathcal{R}^0)'(tu)=0$, then 
\begin{equation*}
	0<(\mathcal{R}^0)''(tu)=\frac{p+1}{t}\Big((\mathcal{R}^1)'(tu)-\frac{p+2}{p+1}(\mathcal{R}^0)'(tu)\Big)=\frac{p+1}{t}(\mathcal{R}^1)'(tu).
\end{equation*}
From this and part \textbf{(i)} (note that $\mathcal{R}^P(t_0(u)u)<\mathcal{R}^0(t_0(u)u)=\mathcal{R}^1(t_0(u)u)$) we get the conclusion of part \textbf{(iii)}.
\QED
\end{proof}

\begin{cor}\label{corPG}
\begin{description}
\item[(i)] If $\lambda<\lambda _{1P}^{D_T}$ and $\Phi_{\lambda,T}^{\prime }(u)=0$, then
$P_{\lambda,T}(u)>0$.
\item[(ii) ] For any $\lambda>\lambda _{1P}^{D_T}$, there exists $u \in
W^{1,2}_{supp}(D_T) \setminus \{0 \}$ such that $\Phi_{\lambda,T}^{\prime }(u)=0$ and $%
P_{\lambda,T}(u)< 0$.
\end{description}
\end{cor}
\begin{proof}  Let us prove \textbf{(i)}. Since $\Phi_{\lambda,T}^{\prime }(u)=0$, then $\mathcal{R}^1(u)=\lambda$. Arguing by contradiction, suppose $P_{\lambda,T}(u)\leq 0$. Then $\mathcal{R}^P(u)\leq \lambda=\mathcal{R}^1(u)$, and therefore by Lemma \ref{propEst}, $1 \geq t_{1,P}(u)>t_1(u)$. Note that $(\mathcal{R}^1)'(tu)>0$ for $t>t_1(u)$. Hence $\mathcal{R}^1(u)\geq \lambda _{1P}^{D_T}(u)>\lambda$ from which we get a contradiction. \\We now prove \textbf{(ii)}. From $\lambda>\lambda _{1P}^{D_T}$, we deduce that there exists $v \in
W^{1,2}_{supp}(D_T) \setminus 0 $ such that $\lambda>\lambda _{1P}^{D_T}(v)=\mathcal{R}^1_v(t_{1,P}(v))$. Therefore there exists $\theta>t_{1P}(v)$ such that $\mathcal{R}^1(\theta v)=\lambda$. We have also from Lemma \ref{propEst}  that $\mathcal{R}^P(\theta v)< \mathcal{R}^1(\theta v)=\lambda$ from which we get $P_{\lambda, T}(\theta v)< 0$. Setting $u=\theta v$ we complete the proof of assertion \textbf{(ii)}.\QED
\end{proof}
	%
\begin{cor}
	$\lambda _{1P}^{D_T}<\lambda _{0}^T$.
\end{cor}
\begin{proof} 
By Lemma \ref{propEst} \textbf{(iii)}, for any $u \in
W^{1,2}_{supp}(D_T) \setminus 0 $, $t_1(u)<t_{1,P}(u)<t_0(u)$. In addition, for any $u \in
W^{1,2}_{supp}(D_T) \setminus 0$, $\mathcal{R}^0(t_0(u)u)=\mathcal{R}^1(t_0(u)u)$. Therefore,
since $(\mathcal{R}^1)'(tu)>0$ for $t \in  (t_1(u),t_0(u))$,  we have 
\begin{equation*}
	\lambda^T_0(u)=\mathcal{R}^0(t_0(u)u)=\mathcal{R}^1(t_0(u)u)>\mathcal{R}^1(t_{1P}(u)u)=\mathcal{R}^P(t_{1P}(u)u).
\end{equation*}
By Lemma \ref{prop:EXlam0}, there exists a  minimizer $\hat{u}_0^T \in  W_{per,0}^{1,2}(D_T)\setminus 0$ of $\lambda _{0}^T(u)$. We thus have
$$
\lambda _{0}^T=\mathcal{R}^0(t_0(\hat{u}_0^T)\hat{u}_0^T)>\mathcal{R}^P(t_{1P}(\hat{u}_0^T)\hat{u}_0^T)\geq \lambda _{1P}^{D_T}.
$$ \QED
\end{proof}
\begin{cor}\label{nonexist} 
Let $\Omega $ be a bounded star-shaped domain in $\mathbb{R}%
^{N}$ with $C^{2}$-manifold boundary $\partial \Omega $. Then for any $%
\lambda <\lambda _{1P}^{D_T}$ problem  \eqref{1}-\eqref{4}  cannot have a
weak solution.
\end{cor}
\begin{proof} 
 We argue by contradiction. Suppose that there exists a non zero weak solution $u \in W_{per,0}^{1,2}(D_T) $ of \eqref{1}-\eqref{4} for some $\lambda<\lambda _{1P}^{D_T}$.   Then $\Phi_{\lambda,T}^{\prime }(u)=0$. By the regularity solutions of elliptic problems (see \cite{giltrud}) it follows that
$u \in C^{1,\gamma}(\overline{D_T})\cap C^{2}({D_T})$
for some $\gamma\in (0,1)$.  Hence,  Corollary \ref{corPG} implies that $P_{\lambda,T}(u)>0$ which yields a contradiction on account of Corollary \ref{cor1}.
	%
\QED
\end{proof}

\section{Minimization problem with Pohozaev's function as a constraint}\label{sec:4}

Let $\lambda\geq \lambda _{1P}^{D_T}$ and recall that
$$
M_{\lambda,T}:=\{w \in W_{per,0}^{1,2}(D_T)\setminus 0: ~\Phi_{\lambda,T}'(u)=0, ~ P_{\lambda,T}(u)\leq 0 \}.
$$ Consider the following constrained minimization problem:
\begin{equation} \label{min1} 
	\hat{\Phi}_{\lambda,T}:=\min_{u\in M_{\lambda,T}} \Phi_{\lambda,T}(u).
\end{equation} 
\begin{prop}
	$M_{\lambda,T}\neq \emptyset$ if and only if $\lambda\geq \lambda _{1P}^{D_T}$.
\end{prop}
\begin{proof} Let $\lambda<\lambda _{1P}^{D_T}$ and $u\in W^{1,2}_{per,0}(D_T)\setminus 0$ such that $\Phi_{\lambda,T}'(u)=0$. Then by Corollary \ref{corPG},  
$P_{\lambda,T}(u)>0$.  Hence  $M_{\lambda,T}=\emptyset$ for any $\lambda<\lambda _{1P}^{D_T}$. \par

Let $\lambda>\lambda _{1P}^{D_T}$. By Corollary \ref{corPG} there is $u \in W_{per,0}^{1,2}(D_T) \setminus 0$ such that $\Phi_{\lambda,T}^{\prime }(u)=0$ and 
$P_{\lambda,T}(u)< 0$, and therefore $u \in M_{\lambda,T}$. Now consider the case $\lambda=\lambda _{1P}^{D_T}$. Let $(\lambda_n)_{n\in\mathbb{N}}$ be a sequence such that $\lambda_n\downarrow \lambda _{1P}^{D_T}$ as $n\to \infty$. Thus, there exists $(u_n)_{n\in \mathbb{N}}\subset W_{per,0}^{1,2}(D_T)\setminus 0 $ satisfying $\Phi_{\lambda_n,T}'(u_n)=P_{\lambda_n,T}(u_n)=0$. Arguing as in Lemma 9 in \cite{IlEg}, it is not difficult to show that $(u_n)$ is bounded in $W_{per,0}^{1,2}(D_T)$ and up to a subsequence $u_n\to \tilde u$ weakly in $W_{per,0}^{1,2}(D_T)\setminus 0$. Therefore passing to the limit as $n\to\infty$, we get $\Phi_{\lambda _{1P}^{D_T}, T}'(\tilde u)\leq 0$ and $P_{\lambda _{1P}^{D_T},T}(\tilde u)\leq 0$. We now redo the same argument as for $\lambda>\lambda _{1P}^{D_T}$ to get $M_{\lambda _{1P}^{D_T},T}\neq \emptyset$.\QED
\end{proof}
\noindent From here it follows that
\begin{cor}
$\hat{\Phi}_{\lambda,T}<+\infty$ for any $\lambda\geq  \lambda _{1P}^{D_T}$.
\end{cor}
\begin{lem}\label{le1e}
For any $\lambda\geq \lambda _{1P}^{D_T}$,  there exists a minimizer $u_\lambda$ of problem (\ref{min1}), i.e.,
$\Phi_{\lambda,T}(u_\lambda)=\hat{\Phi}_{\lambda,T}$ and $u_\lambda\in
M_{\lambda,T}$.
\end{lem}
\begin{proof} Let $\lambda\geq\lambda _{1P}^{D_T}$. Then by the above $M_{\lambda,T}\neq \emptyset$ and $\hat{\Phi}_{\lambda,T}<+\infty$.
Observe that $M_{\lambda,T}$
is bounded. Indeed, if $u\in M_{\lambda,T }$, 
then
\begin{align*}
	c_1\Vert u\Vert _{1}^2 	&\leq \frac{1}{2^{\ast }}\int_{D_T }|\nabla_x u|^{2}\,\mathrm{d}x \mathrm{d}z+\frac{1}{{%
	2}}\int_{D_T }|u_z|^{2}\,\mathrm{d}x \mathrm{d}z+\frac{1}{{%
	q +1}}\int_{D_T}|u|^{{q +1}}\,\mathrm{d}x \mathrm{d}z\\
	&\leq \lambda \frac{1}{%
	{p +1}}\int_{D_T }|u|^{{p +1}}\,\mathrm{d}x\mathrm{d}z\leq c_2\lambda \frac{1}{%
	{p +1}}\Vert u\Vert _{1}^{p +1}
\end{align*}
with some constants $c_1,c_2$ which do not depend on $u\in M_{\lambda,T }$.
Hence, since $p +1<2$, we have $\Vert u\Vert _{1}\leq C<+\infty $, $\forall u\in M_{\lambda,T }$,  where $C<+\infty$ does not depend on $u\in M_{\lambda,T }$.

Let $(u_{m})$ be a minimizing sequence of (\ref{min1}):
\begin{equation*}
\Phi_{\lambda,T}(u_m) \to \hat{\Phi}_{\lambda,T}~~\mbox{as}~~m\to
\infty~~\mbox{and}~ u_m \in M_{\lambda,T},~m\in\mathbb{N}.
\end{equation*} 
Since $(u_{m})$ is
bounded, there exists a subsequence, which we denote again $(u_{m})$, such that
 $u_{m}\rightharpoonup u_{\lambda}$ weakly in $W_{per,0}^{1,2}(D_T)$ and strongly $%
u_{m}\rightarrow u_\lambda$ in $\mathcal{L}^\gamma(-T,T)$, $1<\gamma<2^{\ast }$ for some $u_\lambda \in W_{per,0}^{1,2}(D_T)$. We claim that $%
u_{m}\rightarrow u_\lambda$ strongly in $W_{per,0}^{1,2}(D_T)$. If not, $\Vert u_\lambda\Vert
_{1}<\liminf_{m\rightarrow \infty }\Vert u_{m}\Vert _{1}$ and this implies that
\begin{align*}
\Phi_{\lambda,T}^{\prime }(u_\lambda)< 
 \liminf_{m\rightarrow \infty }\Phi_{\lambda,T}^{\prime }(u_{m}) =0.
\end{align*}%
 Hence $\Phi_{\lambda,T}^{\prime }(u_\lambda)< 0$ and $u_\lambda\neq 0$. Then there exists $\kappa >1$ such
that $\Phi_{\lambda,T}^{\prime }(\kappa u_\lambda)=0$ and $\Phi_{\lambda,T}(\kappa
u_\lambda)<\Phi_{\lambda,T}(u_\lambda)<\hat{\Phi}_{\lambda,T}$. By Proposition \ref{pradd}, $%
\Phi_{\lambda,T}^{\prime }(\kappa u_\lambda)=0$ implies $P_{\lambda }^{\prime
}(\kappa u_\lambda)<0$. From this and since
\begin{equation*}
P_{\lambda }(u_\lambda)<\liminf_{m\rightarrow \infty }P_{\lambda }(u_{m})\leq 0,
\end{equation*}%
we conclude that $P_{\lambda }(\kappa u_\lambda)<0$. Thus $\kappa u_\lambda\in
M_{\lambda }$ and $\Phi_{\lambda,T}(\kappa u_\lambda)<\hat{\Phi}_{\lambda,T}$, which is
a contradiction. Hence, $%
u_{m}\rightarrow u_\lambda$ strongly in $W_{per,0}^{1,2}(D_T)$. The property that $\hat{\Phi}_{\lambda,T}>0$ for $\lambda \in (\lambda _{1P}^{D_T}, \lambda _{0}^T )$, and $\hat{\Phi}_{\lambda,T}<0$ for $\lambda \in (\lambda _{0}^T, +\infty)$ yield that $u_\lambda\neq 0$. In the case $\lambda=\lambda _{0}^T$, the proof of $u_\lambda\neq 0$ follows by the same arguments as in the proof of Lemma \ref{prop:EXlam0}. Thus we get that $\Phi_{\lambda,T}(u_\lambda)=\hat{\Phi}_{\lambda,T}$, 
$u_\lambda \in M_{\lambda,T}$.
\end{proof}\QED

\section{Proof of Theorem \ref{thm1} }\label{sec:prTh1}

Let $\lambda\geq\lambda _{1P}^{D_T}$. By Lemma \ref{le1e} there exists a
minimizer $u_\lambda \in W^{1,2}_{per,0}(D_T) \setminus 0$ of (\ref{min1}). This implies that there
exist Lagrange multipliers $\mu_0$, $\mu_1$ $\mu_2$ such that $|\mu_0|+|\mu_1|+|\mu_2|\neq 0$, $\mu_0, \mu_2\geq 0$ and
\begin{eqnarray}\label{eq2}
&&\mu_0D\Phi_{\lambda,T}(u_\lambda)+\mu_1 D\Phi_{\lambda,T}'(u_\lambda)+\mu_2 DP_{\lambda,T}(u_\lambda)=0,\\
&&\mu_2 P_{\lambda,T}(u_\lambda)=0. \label{eq22}
\end{eqnarray}
\begin{prop}\label{Lag}
Assume  $0<q<p<1$ and
$(q ,p)\in \mathcal{E}_{s}(N)$. Let $\lambda>\lambda _{1P}^{D_T}$ and
 $u_\lambda \in W^{1,2}_{per,0}(D_T)$ be a minimizer in (\ref{min1}) such that
$P_{\lambda,T}(u_\lambda)< 0$. Then  $u_\lambda$ is a weak solution of \eqref{1}-\eqref{4}. Furthermore, one may assume that $u_\lambda\geq 0$ in $D_T$. 
\end{prop}

\begin{proof}
	Since $P_{\lambda,T}(u_\lambda)< 0$, equality \eqref{eq22} implies $\mu_2=0$. Moreover, by Lemma  \ref{pro} we have $\Phi_{\lambda,T}''(u_\lambda)>0$. Testing \eqref{eq2} by $u_\lambda$ we get 
	$0=\mu_0\Phi_{\lambda,T}'(u_\lambda)+\mu_1 \Phi_{\lambda,T}''(u_\lambda)$, and consequently $\mu_1 \Phi_{\lambda,T}''(u_\lambda)=0$. Since $\Phi_{\lambda,T}''(u_\lambda)>0$, we get 
	$\mu_1=0$, consequently $\mu_0\neq 0$ . Thus  $D\Phi_{\lambda,T}(u_\lambda)=0$, i.e., $u_\lambda$ is a weak solution of \eqref{1}-\eqref{4}.    Since
$\Phi_{\lambda,T}(|u|)=\Phi_{\lambda,T}(u)$,
$\Phi_{\lambda,T}'(|u|)=\Phi_{\lambda,T}'(u)$, $P_{\lambda,T}(|u|)=P_{\lambda,T}(u)$ for any $u\in W^{1,2}_{per,0}(D_T)$ we may assume that
$u_{\lambda}\geq 0$. \QED
\end{proof}

\noindent Denote
$$
\mathcal{G}_{\lambda}^T=\{u \in M_{\lambda,T}:~\Phi_{\lambda,T}(u_\lambda)=\hat{\Phi}_{\lambda,T}\}. $$
Select a subset in $\mathcal{G}_{\lambda}^T$ which does not contain "isolated points":
\begin{equation}\label{Mtilda}
	\tilde{\mathcal{G}}_{\lambda}^T=\{u_\lambda \in \mathcal{G}_{\lambda}^T:~\exists u_{\lambda_m} \in \mathcal{G}_{\lambda_m}^T, m\in \mathbb{N}~s.t.~u_{\lambda_m}
	\to u_\lambda~\mbox{ in}~ W_{per,0}^{1,2}(D_T)~\mbox{as}~\lambda_m\downarrow \lambda\}.
\end{equation}
Notice that 
$\tilde{\mathcal{G}}_{\lambda}^T$ is a non-empty set for all $\lambda \in [\lambda _{1P}^{D_T}, +\infty)$. Indeed, by Lemma
\ref{app}, if $\lambda \in [\lambda _{1P}^{D_T}, +\infty)$ and $\lambda_m \downarrow \lambda$, then there exist $u_\lambda \in \mathcal{G}_{\lambda}^T$ and  $u_{\lambda_m} \in \mathcal{G}_{\lambda_m}^T$, $m=1,\ldots $ such that 
 $u_{\lambda_m} \to u_\lambda$ strongly in $W_{per,0}^{1,2}(D_T)$ as $m \to +\infty$.  We introduce
\begin{equation}\label{eqZ}
Z:=\{\lambda \in [\lambda _{1P}^{D_T}, +\infty):~~P_{\lambda,T}(u_\lambda)<0,~~\forall  u_\lambda \in \tilde{\mathcal{G}}_{\lambda}^T\}.
\end{equation}
Then we have:
\begin{prop}\label{Zset}  
 $Z$ is a non-empty open subset of $(\lambda _{1P}^{D_T}, +\infty)$ with $\lambda _{1P}^{D_T} \notin Z$. Furthermore, let $\lambda\in Z$ and a sequence $(\lambda_m) \subset [\lambda _{1P}^{D_T}, +\infty)$ such that $\lambda_m\downarrow\lambda$ as $m\to\infty$. Then $\lambda_m\in Z $ for $m$ large enough.
\end{prop}
\begin{proof}
 By Lemma \ref{le1e}, 
for any $\lambda>\lambda_0^T> \lambda _{1P}^{D_T}$ there exists 
$u_\lambda \in M_{\lambda,T}$ such that $\Phi_{\lambda,T}(u_\lambda)=\hat{\Phi}_{\lambda,T}$.  Note that  $\hat{\Phi}_{\lambda,T} < 0$.  Indeed, since $\lambda>\lambda_0^T$, there exists $v\in W_{per,0}^{1,2}(D_T) \setminus 0$ such that $\lambda>\lambda_0^T(v)\geq \lambda_0^T$. Moreover, from \eqref{T_1}, we have $\Phi_{\lambda_0^T(v), T}(t_0(v)v)=0=\Phi'_{\lambda_0^T(v), T}(t_0(v)v)$. Therefore, $\Phi'_{\lambda, T}(t_0(v)v)<0$. Then, there exist $t^*$ and $\tilde t$ verifying $0<t^*<t_0(v)<\tilde t$, $\Phi'_{\lambda, T}(t^*v)=\Phi'_{\lambda, T}(\tilde t v)=0$ and such that for any $t\in (t^*,\tilde t)$ one has $\Phi'_{\lambda, T}(tv)<0$. Thus $\Phi_{\lambda, T}(t^*v)\leq \Phi_{\lambda, T}(t_0(v)v)<0$. Consequently, $t^*v\in M_{\lambda, T}$ and $\hat{\Phi}_{\lambda,T}=\Phi_{\lambda,T}(u_\lambda)<0$. Now, in view of the identity \eqref{PandE}, we obtain  that
$P_{\lambda,T}(u_\lambda)<0$, that is  $\lambda \in Z$. Hence $(\lambda_0^T, +\infty)\subset Z$ and $Z\neq \emptyset$.

Suppose, contrary to our claim, that  $\lambda _{1P}^{D_T} \in Z$. Then  $\Phi_{\lambda _{1P}^{D_T},T}'(u)=0$, $P_{\lambda _{1P}^{D_T},T}(u)<0$, for $u \in \tilde{\mathcal{G}}_{\lambda _{1P}^{D_T}}^T$ which means that $ \mathcal{R}^P(u)<\lambda _{1P}^{D_T}=\mathcal{R}^1(u)$. Hence by Lemma \ref{propEst}, $1>t_{1P}(u)$. Since $(\mathcal{R}^1)'(tu)> 0$ for $t>t_{1P}(u)$, $\lambda _{1P}^{D_T}(u)=\mathcal{R}^1(t_{1P}(u)u)< \mathcal{R}^1(u)=\lambda _{1P}^{D_T}$. However, this contradicts the definition  \eqref{PPoh2} of $\lambda _{1P}^{D_T}$.

 Let $\tilde{\lambda} \in Z$. Suppose by contradiction that there exists a sequence $(\lambda_m) \subset (\lambda _{1P}^{D_T}, +\infty)\setminus Z$ such that $\lambda_m \downarrow \tilde{\lambda}$ as $m\to \infty$.
This means that there exists a sequence $u_{\lambda_m} \in  \tilde{\mathcal{G}}_{\lambda_m}^T$ such that $P_{\lambda_m,T}(u_{\lambda_m})=0$. Then by  
  Lemma \ref{app}, there exist $u_{\tilde{\lambda}} \in \tilde{\mathcal{G}}_{\tilde{\lambda}}^T$ and a subsequence $(u_{\lambda_m})$, still denoted by $(u_{\lambda_m})$,  such that
$u_{\lambda_m} \to u_{\tilde{\lambda}}$ strongly in 
 $W_{per,0}^{1,2}(D_T)$ as $m \to +\infty$. Hence $P_{\tilde{\lambda},T}(u_{\tilde{\lambda}})=0$, which  contradicts   $\tilde{\lambda} \in Z$.
 
 Now, we will prove that $Z$ is an open set. On the contrary assume there exists $\lambda\in Z$ and sequence $(\lambda_m)\subset (\lambda _{1P}^{D_T}, +\infty)\setminus Z$ such that $\lambda_m\to \lambda$, as $m\to \infty$. Then, the proof follows exactly as above by using Remark \ref{remapp} instead of Lemma \ref{app}.   \QED
\end{proof}
\noindent We introduce
\begin{equation*}
\lambda^*(T):=\sup \{\lambda \in [\lambda _{1P}^{D_T}, +\infty):~~P_{\lambda,T}(u_\lambda)=0,~~\exists  u_\lambda \in \tilde{\mathcal{G}}_{\lambda}^T\}.
\end{equation*}
From the above we can conclude that $\lambda^*(T) \not\in Z$ and  $(\lambda^*(T), +\infty) \subseteq Z$.
\begin{lem}\label{Pr2} For any $T>0$, 
\begin{enumerate}
	\item there exists a minimizer $u_{\lambda^*(T)}^T$ of (\ref{min1}) with  compact support in $\Omega$, 
		\item $u_{\lambda^*(T)}^T$ is a least energy solution of \eqref{1}-\eqref{4} with
$\lambda=\lambda^*(T)$ and $u_{\lambda^*(T)}^T\geq 0$ in $\Omega$,
\item $\lambda _{1P}^{D_T}\leq\lambda^*(T)<\lambda_0^T$.
\end{enumerate}
\end{lem}
\begin{proof} By construction $\lambda^*(T) \in (\lambda _{1P}^{D_T}, +\infty)$, and therefore, $\tilde{\mathcal{G}}_{\lambda^*(T)}^T \neq \emptyset$. Since $\lambda^*(T) \not\in Z$,  
 there exists $u_{\lambda^*(T)}^T \in \tilde{\mathcal{G}}_{\lambda^*(T)}^T$ such that $P_{\lambda^*(T),T}(u_{\lambda^*(T)}^T)= 0$. By \eqref{Mtilda}, there exists a sequence
$\lambda_m \in Z$, $m=1,\ldots $ such that $\lambda_m \downarrow
\lambda^*(T)$ as $m \to \infty$, and a sequence $u_{\lambda_m} \in \mathcal{G}_{\lambda_m}^T$, $m=1,2,\ldots$, such that $u_{\lambda_m} \to u_{\lambda^*(T)}^T$ strongly in $ W_{per,0}^{1,2}(D_T)$ as $m\to +\infty$.  

Since $\lambda_m \in Z$,  $P_{\lambda_m,T}(u_{\lambda_m})<0$, $m=1,2,\ldots$, and thus,  Proposition \ref{Lag} implies that 
$u_{\lambda_m}$, $m=1,2,\ldots$ are weak nonnegative solutions of \eqref{1}-\eqref{4}.   This and the strong convergence  $u_{\lambda_m} \to u_{\lambda^*(T)}^T$ 
 in  $W_{per,0}^{1,2}(D_T)$ yield that $u_{\lambda^*(T)}^T$ is a non-negative weak solution of \eqref{1}-\eqref{4}. 
Then from elliptic regularity theory (see \cite{giltrud})  
 one gets that
$u_{\lambda^*(T)}^T\in C^{1,\gamma}(\overline{D_T})\cap C^{2}({D_T})$
for some $\gamma\in (0,1)$. Hence there holds the equality $P_{\lambda^*(T),T}(u_{\lambda^*(T)}^T)= 0$ which implies from Corollary \ref{cor1} that $u_{\lambda^*(T)}^T$ has a  compact support in $\Omega$, and since  $u_{\lambda^*(T)}^T$ is a  minimizer of (\ref{min1}), $u_{\lambda^*(T)}^T$ is a least energy solution of \eqref{1}-\eqref{4}.

To conclude the proof, it is sufficient to show that $\lambda^*(T)<\lambda_0^T$. From the proof of Proposition \ref{Zset} we have already $\lambda^*(T)\leq \lambda_0^T$. Let us assume that $\lambda^*(T)=\lambda_0^T$. By Lemma \ref{prop:EXlam0}, there exists $\hat{u}_0^T$ such that $D\Phi_{\lambda^*(T),T}(\hat{u}_0^T)=0$, $\Phi_{\lambda^*(T),T}(\hat{u}_0^T)=0$. Then $\hat{u}_0^T \in M_{\lambda^*(T),T}$ and therefore,  $\Phi_{\lambda^*(T),T}(u_{\lambda^*(T)}^T)\leq \Phi_{\lambda^*(T),T}(\hat{u}_0^T)=0$. But  $u_{\lambda^*(T)}^T$ is a solution with compact support, and therefore $P_{\lambda^*(T),T}(u_{\lambda^*(T)}^T)=0$. This by equality \eqref{PandE} implies the opposite inequality $\Phi_{\lambda^*(T),T}(u_{\lambda^*(T)}^T)> 0$.
\QED
\end{proof}

\noindent\textbf{Proof of Theorem \ref{thm1}}: By Lemma \ref{Pr2}, we have $\lambda^*(T)\in \textcolor[rgb]{1,0,0}{(\lambda _{1P}^{D_T},\lambda_0^T)}$. Moreover, from Lemma \ref{le1e}, for all $\lambda\geq\lambda^*(T)$, there exists a minimizer $u_\lambda^T$ for $\Phi_{\lambda,T}$ such that $P_{\lambda,T}(u_\lambda^T)\leq 0$. Since $(\lambda^*(T),+\infty)\subset Z$, for all $\lambda\geq\lambda^*(T)$ with $\lambda\neq\lambda_{0}^T$, we can find a minimizer $u_\lambda^T$ such that $P_{\lambda,T}(u_\lambda^T)<0$. Then, on account of Proposition \ref{Lag}, we deduce that $u_\lambda^T$ is a weak solution of \eqref{1}-\eqref{4}. 
To complete the prove of $(1^o)$ for $\lambda\geq\lambda^*(T)$, we observe that 
 \begin{itemize}
 	\item[(a)] From Proposition \ref{propL0} (i), for $\lambda\in [\lambda^*(T),\lambda_{0}^T)$, we have $\Phi_{{\lambda},T}(u_\lambda^T)>0$.
 	\item[(b)] For $\lambda=\lambda_{0}^T$, on account of Lemma \ref{prop:EXlam0}, we have the existence of a weak solution satisfying $\Phi_{{\lambda},T}(u_\lambda^T)=0$.
 	\item[(c)] 
 	From the proof of Proposition \ref{Zset},  we have $\hat{\Phi}_{\lambda,T}<0$, that is, $\Phi_{\lambda,T}(u_\lambda^T)<0$, for any $\lambda>\lambda_0^T$. 
 \end{itemize} 
 Lemma \ref{pro} implies that $\Phi_{{\lambda},T}''(u_\lambda^T)>0$.  As above, from elliptic regularity theory  we have $u_{\lambda}^T\in C^{1,\gamma}(\overline{D_T})\cap C^{2}({D_T})$
for some $\gamma\in (0,1)$. 
  The proof of $(2^o)$ is clear from parts $1$ and $2$ of Lemma \ref{Pr2}. Subsequently, for $\lambda\in [\lambda_{0}^T,+\infty)$, the least energy solution satisfies $\Phi_{\lambda,T}(u_{\lambda}^T)\leq 0$, consequently, $P_{\lambda,T}(u_\lambda^T)<0$. This proves $(3^o)$. The proof of $(4^o)$ follows from Corollary \ref{nonexist}. \QED

\section{Proof of Theorem \ref{thm2}}\label{sec:prTh2}

\textit{Proof of $(1^o)$.} Notice that $\lambda^T_{1P}\leq \lambda_*(\Omega)$, where $\lambda_*(\Omega)$ is defined in accordance to \eqref{sigma_lambda}. Indeed, for any $\lambda\geq \lambda_*(\Omega)$ problem \eqref{1}-\eqref{4} has a solution $u_\lambda(z,x)\equiv u^N_\lambda(x)$, $z \in [-T,T]$, where $ u^{N}_{\lambda}$ is a least energy solution of ($P_\lambda(\Omega)$). On the other hand, by Corollary \ref{nonexist},  \eqref{1}-\eqref{4} has no solutions if $\lambda\leq  \lambda^T_{1P}$.
Furthermore, by \eqref{lamomeg} we have $\lambda^*(T)<\lambda^T_0\leq \lambda^\Omega_0$, $\forall T>0$.
Here  
$$
\lambda^\Omega_0=c_{0}^{q ,p }\inf_{u \in W^{1,2}_0(\Omega)\setminus 0} \frac{\left(\int_{\Omega}|\nabla u|^{{2}}\mathrm{d}x\right)^{\frac{p-q}{1-q}}\left(\int_{\Omega }|u|^{{q +1}}\mathrm{d}x\right)^{\frac{1-p}{1-q}}}{\int_{\Omega}|u|^{{p +1}}\mathrm{d}x}.
$$
 Let us show that
\begin{equation}\label{barLamb}
 \bar{\lambda}^*:=\sup_{T>0}\{\lambda^*(T)\}<\lambda^\Omega_0.
\end{equation}
Indeed, suppose this is false. Then there exists $T_k \to+\infty$ such that $\lambda^*(T_k)
\to \lambda^\Omega_0$, and consequently  $\lambda^{T_k}_0
\to \lambda^\Omega_0$. This yields
$$
\lambda_0^{D_\infty}=\lambda^\Omega_0.
$$
Similar to the proof of Lemma \ref{prop:EXlam0} it can be shown that there exists a minimizer  $\hat{u}_0 \in W^{1,2}_0(\Omega)$  of $\lambda^\Omega_0(u)$, i.e., $\lambda^\Omega_0=\lambda^\Omega_0(\hat{u}_0 )$. Take $\psi_K \in C^\infty_{sym}(\mathbb{R})$, $K \geq 1$ such that $\psi_K(s)=1$ if $|s|<K-1$, $\psi_K(s)=0$ if $|s|>K-1/2$, and $\max_{s \in \mathbb{R}}(|\psi_K(s)|+|\psi_K'(s)|)<C_0<+\infty$, $\forall K\geq 1$ for some $C_0$ independent of $K$.
Consider $\phi_K(z,x):=\psi_K(z)\hat{u}_0(x)$, $K\in\mathbb{N}$. Evidently, $\phi_K \in W^{1,2}_{sym}(D_\infty)$, $K\in\mathbb{N}$. Note that for $\gamma \in [1,2^*]$, 
$$
\int_{D^\infty}|\phi_K|^{\gamma}\,\mathrm{d}x\mathrm{d}z=2(K-1)\int_\Omega |\hat{u}_0|^\gamma \,\mathrm{d}x+2\int_{K-1}^{K-1/2}|\psi_K(z)|^\gamma \mathrm{d}z\int_\Omega |\hat{u}_0|^\gamma \,\mathrm{d}x,~~K\in\mathbb{N}.
$$ 
Consequently, 
$$
\frac{1}{2(K-1)}\int_{D^\infty}|\phi_K|^{\gamma}\,\mathrm{d}x\mathrm{d}z \to \int_\Omega |\hat{u}_0|^\gamma \,\mathrm{d}x~~\mbox{as}~~K \to +\infty.
$$
Similarly
$$
\frac{1}{2(K-1)}\int_{D^\infty}|\nabla \phi_K|^2\,\mathrm{d}x\mathrm{d}z \to \int_\Omega |\nabla \hat{u}_0|^\gamma \,\mathrm{d}x~~\mbox{as}~~K \to +\infty.
$$
Since $\lambda_0^{D_\infty}=\lambda^\Omega_0$, this implies that $\lambda_0^{D_\infty}(\phi_K) \to \lambda_0^{D_\infty}$ as $K \to +\infty$, that is $(\phi_K)$ is a minimizing sequence of
$\lambda_0^{D_\infty}(u)$. Observe that $t_K:=\|\phi_K\|_1 \to +\infty$ as $K \to +\infty$. 
Consider $v_K:=\phi_K/t_K$, $k=1,2,\ldots$ . Then $\|v_K\|_1 $, $K=1,2,\ldots$  is bounded and due to homogeneity of  $\lambda_0^{D_\infty}(u)$, $\|v_K\|_1 $, $K=1,2,\ldots$ is a minimizing sequence of $\lambda_0^{D_\infty}(u)$. Thus by 
Corollary \ref{cor:EXlam0} there exists  strong in  $W^{1,2}_{sym}(D_\infty)$ non-zero limit point of 	$(v_K)_{K=1}^\infty$. However, it is easy to see $v_K \to 0$ as $K \to +\infty$. We get a contradiction and thus \eqref{barLamb} is true.

Thus, to conclude the proof of $(1^o)$, it is sufficient to show the following.\\
\textit{Claim:}
\textit{For any $\epsilon\in (0, \lambda^\Omega_0-\bar{\lambda}^*)$, there exists $T_\epsilon> 0$ such that for any $T>T_\epsilon$ and $\lambda\in [\lambda^*(T), \lambda^\Omega_0-\epsilon)$ any   least energy solution $u_{\lambda}^T$ of \eqref{1}-\eqref{4}  is periodically nontrivial.} \par
 To prove the claim, suppose on the contrary, that there exist $\epsilon\in (0, \lambda^\Omega_0-\bar{\lambda}^*)$, sequences $(T_m)$, $T_m \to +\infty$ as $m\to +\infty$ and  $\lambda_m \in [ \lambda^*(T_m), \lambda^\Omega_0-\epsilon)$, $m=1,2,\ldots$, such that for any $m=1,2,\ldots$, there exists a 		least energy solution $u_{\lambda_m}^{T_m}$  of \eqref{1}-\eqref{4} which is periodically trivial. Then  obviously, $u_{\lambda_m}^{T_m}(z,\cdot)\equiv u^{N}_{\lambda_m}(\cdot) $,  $\forall z \in [-T_m,T_m]$ and $\lambda^{T_m}_0=\lambda^\Omega_0$, $m=1,2,\ldots$, where $ u^{N}_{\lambda_m}$ is a least energy solution of ($P_{\lambda_m}(\Omega)$), $m=1,2,\ldots$ (see Appendix \ref{A}). Furthermore, $\lambda_m \geq \lambda_*(\Omega)$, $m=1,2,\ldots$,  since by Lemma \ref{lemDHI}, ($P_\lambda(\Omega)$) has no solution for $\lambda< \lambda_*(\Omega)$.  Hence, $0<\lambda_*(\Omega)\leq \lambda_m<  \lambda^\Omega_0-\epsilon<+\infty$, $m=1,2,\ldots$, and therefore there exists a limit point
		$\bar{\lambda} \in [\lambda_*(\Omega),\lambda^\Omega_0-\epsilon]$ such that $\lim_{m\to +\infty}\lambda_m=\bar{\lambda}$.
Since $\bar{\lambda}<\lambda^0_\Omega$, we have $\Phi_{\bar{\lambda}}(u^{N}_{\bar{\lambda}})>\delta_\epsilon>0$ for some $\delta_\epsilon$. Hence $\Phi_{{\lambda_m}}(u^{N}_{\lambda_m})>\delta_\epsilon/2>0$, for sufficiently large $m$. Therefore,
	\begin{equation}\label{inftym}
		\Phi_{\lambda_m,T_m}(u_{\lambda_m}^{T_m})=2T_m \Phi_{{\lambda_m}}(u^{N}_{\lambda_m})\to +\infty~~\mbox{as}~~ m \to +\infty.
	\end{equation}
 Since $\bar{\lambda}\geq  \lambda_*(\Omega)$, Proposition \ref{prop7.2} yields that there exist $T_{\bar{\lambda}}>0$ and finite support function $\phi_{\bar{\lambda}} \in C^\infty_0(D_{T_{\bar{\lambda}}})$ such that $\phi \in M_{\bar{\lambda},T}$ for any $T>T_{\bar{\lambda}}$. This implies that 
	$$
	\hat{\Phi}_{\bar{\lambda},T}\leq \Phi_{\bar{\lambda},T}(\phi_{\bar{\lambda}})<+\infty, ~~\forall T>T_{\bar{\lambda}},
	$$
which contradicts \eqref{inftym}.

\textit{Proof of $(2^o)$.}  Suppose, contrary to our claim,  that there is a sequence $(T_m)$, $T_m \to 0$ as $m\to +\infty$ such that  for every $m=1,2,\ldots$, 
 there exists a least energy solution $u_{\lambda^*(T_m)}$ of \eqref{1}-\eqref{4} which has a compact support in $\mathbb{R}^{N+1}$. By Appendix A, problem $(P(D_{T_m}))$ has no compact support solution for $\lambda<\lambda_{\ast }(D_{T_m})$. Hence 
	$\lambda^*(T_m)\geq \lambda_{\ast }(D_{T_m})$. This by \eqref{InfRad} implies that  $\lambda^*(T_m) \to +\infty$.

	Notice that $W^{1,2}_0(\Omega)$ can be identified with $W_c:=\{u \in 
	W_{\Omega,per}^{1,2}(D_T):~\int_{D_T}  |u_z|^{2}\,\mathrm{d}x\mathrm{d}z =0\}$ and 
	\begin{align*}
		\lambda_{0}^T(u)&=\frac{(2T\int_{\Omega}|u|^{{q +1}}\mathrm{d}x)^{\frac{1-p }{%
		1-q }}(2T\int_{\Omega}|\nabla_x u|^{2}\mathrm{d}x)^{\frac{p -q }{1-q }%
		}}{2T\int_{\Omega }|u|^{{p +1}}\mathrm{d}x}\\
		&=\frac{(\int_{\Omega}|u|^{{q +1}}\mathrm{d}x)^{\frac{1-p }{%
		1-q }}(\int_{\Omega}|\nabla_x u|^{2}\mathrm{d}x)^{\frac{p -q }{1-q }%
		}}{\int_{\Omega }|u|^{{p +1}}\mathrm{d}x}=:\lambda _{0}^\Omega(u),~~\forall u \in W_c, \forall T>0.
	\end{align*}
	Hence  
		\begin{equation}\label{lamomeg}
		\lambda_{0}^T=\inf_{u\in W_{per,0}^{1,2}(D_T)\setminus 0}\lambda_{0}^T(u)\leq \inf_{u\in W^{1,2}_0(\Omega)\setminus 0}\lambda _{0}^\Omega(u)=:\lambda_{0}^{\Omega}, ~~\forall T>0,
	\end{equation}
		and thus by Lemma \ref{Pr2} we have $\lambda_*(D_{T_m})< \lambda_{0}(D_{T_m})\leq \lambda_{0}^{\Omega}<+\infty$ for all $m=1,2,\ldots$, which means a contradiction. \QED

\section{Conclusions and open problems}\label{sec:concl}
 We proved for the equation  with non-Lipschitz non-linearity the existence of least energy solutions  periodic in one variable and subject to the zero Dirichlet conditions on $\partial \Omega$ for other variables. 
Moreover, we  find an upper threshold $\lambda _{1P}^{D_T}$ for the nonexistence of  solutions of the problem. We believe that the point $\lambda^*(T)$ in Theorem \ref{thm1} is a limit value for the existence of nonnegative solutions of the problem and it corresponds to the turning point bifurcation of a branch of the nonnegative solutions.   Note that the main difficulty that  had been overcome in this result is obtaining solutions with positive energy $\Phi_{\lambda,T}(u_{\lambda}^T)>0$ for $\lambda \in [\lambda^*(T), \lambda^T_0)$. Apparently, the least energy solutions for $\lambda \in (\lambda^T_0, +\infty)$ can be obtained without assumption $N(1-q
)(1-p )-2(1+q )(1+p )>0$, by direct application of the Nehari manifold method. Furthermore, we expect that if $\lambda \in [\lambda^T_0, +\infty)$, then it can be obtained (perhaps by the mountain pass theorem) the second branch of nonnegative solutions \eqref{1}-\eqref{4}. However, we do not know whether it is possible to construct the second branch of solutions for $\lambda \in (\lambda^*(T), \lambda^T_0)$ and whether it forms with  the branch of nonnegative solutions $u_{\lambda}^T$   a  turning point bifurcation at the value $\lambda^*(T)$. 
 
Our second result addresses the possibility of the existence of compactly supported solutions of \eqref{1}-\eqref{4} with respect to part of the variables $x \in \mathbb{R}^N$.
Theorems \ref{thm1}, \ref{thm2} confirm a positive answer. However, it does not give a complete answer as to whether they are new type solutions with compact supports for equation \eqref{1}. 
In fact, we conjecture that there exist $T_1^*, T_2^*$, $0<T_1^*<T_2^*< +\infty$ such that 
	\begin{itemize}
	\item for each $T\in (0,T_1^*)$, any  compactly supported  least energy solution $u_{\lambda^*(T)}^T$ of \eqref{1}-\eqref{4} is periodically trivial, i.e. $\int_{D_T }|(u_{\lambda^*(T)})_z|^{2}\mathrm{d}x\mathrm{d}z= 0$. 
		\item for any $T\in (T_1^*,T_2^*)$  the compactly supported  least energy solution $u_{\lambda^*(T)}^T$ of \eqref{1}-\eqref{4} is periodically non-trivial and has no compact support in $D_T$. 
		\item for any $T\geq T_2^*$, any  compactly supported periodic least energy solution $u_{\lambda^*(T)}^T$ of \eqref{1}-\eqref{4} is, in fact, a compactly supported solution of $(1_{\mathbb{R}^{N+1}}^*)$.
	\end{itemize}
	 
An interesting question that also remains open is whether compactly supported solutions \,\,  $u_{\lambda^*(T)}^T(z,x)$ are radially symmetric with respect to variables $x \in \mathbb{R}^N$,
as is the case of problem \eqref{PL} by Lemma \ref{lemDHI} (see also \cite{Kaper2, CortElgFelmer-2, Serrin-Zou}).

	%
	
It is worth noting that the obtained results in this article give reason to expect that problem \eqref{1}-\eqref{4} can have periodically non-trivial nonnegative least energy solutions, which partially satisfying the Hopf maximum principle on the boundary $\partial \Omega$. Similar result was recently obtained in \cite{BDI} for another problem.

\section{Acknowledgments:}
The first author is funded by IFCAM (Indo-French Centre for Applied Mathematics) IRL CNRS 3494. The research of the second
author was supported by the Russian Science Foundation
(RSF) (No. 22-21-00580).

\begin{appendices}

\section{Further investigations of Theorem \ref{thm:KKSZ}}\label{A}

We denote by $R_M$ the radius of the supporting ball $B_{R_M}$ of the unique
(up to translation in $\mathbb{R}^{M}$) compactly supported solution $\psi^M$ of \ref{Eqw}.

Let $M\geq 1$ and $\Theta$ be a bounded  domain in $\mathbb{R}^M$.
The largest ball $%
B_{R^*(\Theta )}^a:=\{x\in \mathbb{R}^{M}:~|x-a|\leq R^*(\Theta )\}$ with some  $a \in \Theta$ contained in $\Theta$ is said to be \textit{inscribed ball} in $\Theta$.  In what follows, we always assume that the centre $a$ of such ball coincides with $0 \in \mathbb{R}^M$ and will use the notation $B_{R^*(\Theta )}:=B_{R^*(\Theta )}^0$. Thus,
$$
R^*(\Theta )=\max\{R: B_R \subseteq \Theta\}. 
$$
Let $R\in (0, R^*(\Theta ))$ and $\sigma_R=R/R_M$. Then the function $u^{M}_{\lambda_R
}(y):=(\sigma_R) ^{-\frac{2}{1-q }}\cdot \psi^M(y/\sigma_R )$  is a compactly
supported  with supp$(u^{M}_{\lambda_R
})=B_R \subset \Theta$ solution  of 
\begin{equation}  \label{PL}
\begin{cases}
-\Delta u=\lambda |u|^{p-1}u- |u|^{q-1}u ~~\mbox{in}~\Theta , \\ 
~~u=0~~\mbox{on}~\partial \Theta ,%
\end{cases}
\tag*{$P_\lambda(\Theta)$}
\end{equation}
 where
 \begin{equation*}
	 \lambda=\lambda_R:= (\sigma_R)^{-\frac{2(p -q )}{1-q }}=\left(\frac{R_M}{R}\right)^{\frac{2(p -q )}{1-q }}.
 \end{equation*}
In what follows, we denote $\sigma(\Theta)=R^*(\Theta )/R_M$,
\begin{equation}\label{sigma_lambda}
	  \lambda_*(\Theta):= (\sigma(\Theta))^{-\frac{2(p -q )}{1-q }}=\left(\frac{R_M}{R^*(\Theta )}\right)^{\frac{2(p -q )}{1-q }} \quad\mbox{and }u^{M}_{\lambda_*(\Theta)
}:=u^{M}_{\lambda_{R^*(\Theta )}}.
 \end{equation}
Note that 
\begin{align}\label{InfRad}
\lambda_R \to +\infty~~\mbox{as}~~R \to 0.
 \end{align}
Moreover, if $\lambda<\lambda_*(\Theta)$, then \eqref{PL} has no compactly supported solutions. 
Define 
\begin{equation*}
\Lambda_{1P}^{\Theta}=\inf_{u\in W_{0}^{1,2}(\Theta)\setminus 0}\Lambda_{1P}^{\Theta}(u),  
\end{equation*}%
where
\begin{equation*}
\Lambda^\Theta_{1P} (u)=c_{1P}^{q ,p,0 }\frac{(\int_{\Theta }|u|^{{q +1}}\mathrm{d}x)^{\frac{1-p }{%
1-q }}(\int_{\Theta }|\nabla u|^{2}\mathrm{d}x)^{\frac{p -q }{1-q }%
}}{\int_{\Theta }|u|^{{p +1}}\mathrm{d}x}
\end{equation*}%
and
\begin{equation*}
c_{1P}^{q ,p,0 }=\frac{(p +1)(2^{\ast }-q +1)}{(p
-q )2^{\ast }}\left( \frac{2^{\ast }(p -q )}{(2^{\ast }-p
-1)(q +1)}\right) ^{\frac{p -q }{1-q }}.  
\end{equation*}%
 Note that $\Lambda^\Theta_{1P} (u)$ is different from $\lambda^\Theta_{1P} (u)$ defined by \eqref{LpoT}.
The following result  follows from \cite{DiazHerIlCOUNT}

\begin{lem} \label{lemDHI} Let $M\geq 3$ and $\Theta $ be a bounded strictly
star-shaped domain in $\mathbb{R}^{M}$ with $C^{2}$-manifold boundary $%
\partial \Theta $. Assume that $(q ,p )\in \mathcal{E}_{s}(M)$.
Then $\Lambda _{1P}^{\Theta}<\lambda_{\ast }(\Theta)$ and 
 \par
$(1^o)$ For any $\lambda \geq
\lambda ^{\ast }(\Theta)$ problem \ref{PL} possess least energy solution $u_{\lambda }$. 

\par
$(2^o)$ The least energy solution $u_{\lambda_*(\Theta)}$ is radially symmetric and compactly supported in $\Theta \subset \mathbb{R}^M$. Moreover,   $u_{\lambda_*(\Theta)}\equiv u^{M}_{\lambda_*(\Theta)}$.

\par
$(3^o)$ For any $\lambda <\lambda ^{\ast }(\Theta)$, problem \ref{PL} has no weak solution.
\end{lem}

Let us show
\begin{prop}\label{prop7.2} For any $\lambda\geq \lambda_*(\Omega)$, there exist  $T_{\lambda}>0$ and $\phi \in W^{1,2}_{sym}(D_\infty)$ such that $\pi^T(\phi) \in M_{\lambda,T}$ for any $T>T_{\lambda}$.
\end{prop}
\begin{proof} 
Observe that $C^\infty_0(\Omega)$ can be identified with 
$$
C^\infty_{sym,i}(\mathbb{R} \times \Omega):=\bigg\{u \in C^\infty_{0, sym}(\mathbb{R} \times \Omega) : \int_{D_\infty }|u_z|^{2}\mathrm{d}x\mathrm{d}z=0\bigg\},
$$
and $\lambda_{1P}^{\infty}(u)=\Lambda _{1P}^{\Omega}(I_{\Omega}u)$ for $u \in C^\infty_{sym,i}(\mathbb{R} \times \Omega)$, where $I_{\Omega}u(\cdot,x)=u(0,x)$, $x \in \Omega$.
Consequently,
$$
\lambda_{1P}^{\infty}:=\min\{\lambda_{1P}^{\infty}(u): u \in  W^{1,2}_{sym}(D_\infty)\}\leq \min\{\lambda_{1P}^{\infty}(u): u \in  C^\infty_{sym,i}(\mathbb{R})\}=\Lambda _{1P}^{\Omega}.
$$
By Lemma  \ref{lemDHI}, one has	$\Lambda _{1P}^{\Omega}<	\lambda_*(\Omega)$ and therefore $\lambda_{1P}^{\infty}<	\lambda_*(\Omega)\leq \lambda$. 

Since $\lambda_{1P}^{\infty}<	\lambda$, 
	there is $\phi \in C^\infty_{0, sym}(\mathbb{R} \times \Omega) \setminus 0$ such that $\lambda_{1P}^{\infty}< \lambda_{1P}^{\infty}(\phi)<\lambda $. Since $P_{\lambda_{1P}^{\infty}(\phi)}(\phi)=0$, $\Phi_{\lambda_{1P}^{\infty}(\phi)}'(\phi)=0$, we have $P_{\lambda,{\infty}}(\phi)< 0$, $\Phi_{\lambda,{\infty}}'(\phi)< 0$, and therefore, there exists $t_{\lambda}(\phi)>1$ such that $\Phi_{{\lambda},{\infty}}'(t_{\lambda}(\phi)\phi)=0$. By Proposition \ref{pradd}, we have
$dP_{{\lambda},{\infty}}(t\phi)/dt<0$ for $t\in [1, t_{\lambda}(\phi)]$. Thus $P_{{\lambda},{\infty}}(t_{\lambda}(\phi)\phi)<0$ and $\Phi_{{\lambda},{\infty}}'(t_{\lambda}(\phi)\phi)=0$.	
Since $\phi$ has a compact support in $D_\infty:=\mathbb{R} \times \Omega$, there exists $T_{\lambda}>0$ such that $\forall\, T>T_{\lambda}$,
\begin{align*}
		& P_{{\lambda},T}(t_{\lambda}(\phi)\phi)=P_{{\lambda},{\infty}}(t_{\lambda}(\phi)\phi)<0, \\
		&		\Phi_{{\lambda},T}'(t_{\lambda}(\phi)\phi)=\Phi_{{\lambda},{\infty}}'(t_{\lambda}(\phi)\phi)=0,~~
\end{align*}
 which means that $\pi^T(t_{\lambda}(\phi)\phi )\in M_{\lambda,T}$ for any $T>T_{\lambda}$.\QED
\end{proof}


%

%
\section{A technical result about minimizers}
\begin{lem}
\label{app} Assume $\lambda \in [\lambda _{1P}^{D_T}, +\infty)$ and $u_{\lambda_m}$
is a sequence of minimizers of (\ref{min1}), where $\lambda_m \downarrow \lambda$ as
$m\to +\infty$. Then there exist a minimizer $u_\lambda$ of (\ref{min1}) and
a subsequence, still denoted by $(u_{\lambda_m})$, such that $u_{\lambda_m}
\to u_\lambda$ strongly in $ W_{per,0}^{1,2}(D_T)$ as $m \to +\infty$.
\end{lem}

\begin{proof} Let $\lambda \in \lbrack \lambda _{1P}^{D_T},+\infty )$%
, $\lambda _{m}\rightarrow \lambda $ as $m\rightarrow +\infty $ and $%
u_{\lambda _{m}}$ be a sequence of minimizers of (\ref{min1}). As in the
proof of Lemma \ref{le1e} it is derived that the set $(u_{\lambda _{m}})$ is
bounded in $ W_{per,0}^{1,2}(D_T)$. Hence by the Sobolev embedding theorem there exists
a subsequence, still denoted by $(u_{\lambda _{m}})$, such that
\begin{equation*}
u_{\lambda _{m}}\rightharpoonup \bar{w}_{\lambda }~~\mbox{weakly in}%
~~ W_{per,0}^{1,2}(D_T),~~~u_{\lambda _{m}}\rightarrow \bar{w}_{\lambda }~~%
\mbox{strongly in}~~\mathcal L^{q}(-T,T),  
\end{equation*}
where $1\leq q<2^{\ast }$, for some limit point $\bar{w}_{\lambda } \in  W_{per,0}^{1,2}(D_T)$. Hence
\begin{align}
	&\Phi_{\lambda,T}(\bar{w}_{\lambda })\leq \liminf_{m\rightarrow \infty
	}\Phi_{\lambda_m,T}(u_{\lambda _{m}}),\label{C1}~~\\
	&\Phi_{\lambda,T}^{\prime }(\bar{w}_{\lambda })\leq \liminf_{m\rightarrow \infty
	}\Phi_{\lambda_m,T}^{\prime }(u_{\lambda _{m}})=0%
	, \nonumber\\
	&P_{\lambda,T}(\bar{w}_{\lambda })\leq \liminf_{m\rightarrow \infty
	}P_{\lambda_m,T}(u_{\lambda _{m}})\leq 0.\nonumber
\end{align}
By Lemma \ref{le1e} there exists a minimizer $%
u_{\lambda }$ of (\ref{min1}), i.e., $u_{\lambda }\in M_{\lambda,T}$ and $\hat{%
\Phi}_{\lambda,T}=\Phi_{\lambda,T}(u_{\lambda })$. Observe
\begin{equation*}
|\Phi_{\lambda,T}(u_{\lambda })-\Phi_{\lambda_m,T}(u_{\lambda })|<C|\lambda
-\lambda _{m}|, ~~m=1,2,\ldots,
\end{equation*}%
where $C<+\infty $ does not depend on $m$. Since $\lambda_m>\lambda$, $m=1,2,\ldots$, $P_{\lambda_m,T}(u_{\lambda})\leq P_{\lambda,T}(u_{\lambda})\leq 0$, and one can find $t_{\lambda_m}(u_{\lambda})>0$  such that $\Phi_{\lambda_m,T}'(t_{\lambda_m}(u_{\lambda})u_{\lambda })=0$ and $\Phi_{\lambda_m,T}''(t_{\lambda_m}(u_{\lambda})u_{\lambda })\geq 0$, $m=1,2,\ldots$. Note that $\Phi_{\lambda_m,T}'(u_{\lambda})< \Phi'_{\lambda,T}(u_{\lambda})= 0$, $m=1,2,\ldots$. Hence by Proposition \ref{pradd}, $P_{\lambda_m,T}'(u_{\lambda})<0$ and $P_{\lambda_m,T}'(t_{\lambda_m}(u_{\lambda})u_{\lambda})<0$, which implies $P_{\lambda_m,T}(t_{\lambda_m}(u_{\lambda})u_{\lambda})<0$. From this we get 
\begin{equation*}
\Phi_{\lambda_m,T}(u_{\lambda })
\geq \Phi_{\lambda_m,T}(t_{\lambda_m}(u_{\lambda})u_{\lambda })
\geq \Phi_{\lambda_m,T}(u_{\lambda _{m}}), ~~m=1,2,\ldots,
\end{equation*}%
 Thus  we have
\begin{equation*}
\Phi_{\lambda,T}(u_{\lambda })+C|\lambda -\lambda _{m}|>\Phi_{\lambda_m,T}(u_{\lambda })\geq \Phi_{\lambda_m,T}(u_{\lambda _{m}})
\end{equation*}%
and consequently, $\hat{\Phi}_{\lambda,T}:=\Phi_{\lambda,T}(u_{\lambda })\geq
\liminf_{m\rightarrow \infty }\Phi_{\lambda_m,T}(u_{\lambda _{m}})$. Hence by %
\eqref{C1} we obtain
\begin{equation*}
\Phi_{\lambda,T}(\bar{w}_{\lambda })\leq \hat{\Phi}_{\lambda,T}.
\end{equation*}%
Assume $\Phi_{\lambda,T}^{\prime }(\bar{w}_{\lambda })<0$. Then $\bar{w}_{\lambda }\neq 0$ and since $\Phi_{\lambda,T}^{\prime }(t_{\lambda}(\bar{w}_{\lambda })\bar{w}_{\lambda })=0$, we have 
$$
\Phi_{\lambda,T}(t_{\lambda}(\bar{w}_{\lambda })\bar{w}_{\lambda })<\Phi_{\lambda,T}(%
\bar{w}_{\lambda })\leq \hat{\Phi}_{\lambda,T}.
$$
Moreover, since $P_{\lambda,T}(\bar{w}_{\lambda })\leq 0$, we have $P_{\lambda,T}(t_{\lambda}(\bar{w}_{\lambda })\bar{w}%
_{\lambda })\leq 0$. Thus $t_{\lambda}(\bar{w}_{\lambda })\bar{w}_{\lambda }\in
M_{\lambda,T}$ and since $\Phi_{\lambda,T}(t_{\lambda}(\bar{w}_{\lambda })\bar{w}%
_{\lambda })<\hat{\Phi}_{\lambda,T}$ we get a contradiction. Hence $\Phi_{\lambda,T}(%
\bar{w}_{\lambda })=\hat{\Phi}_{\lambda,T}$, $\Phi_{\lambda,T}^{\prime }(\bar{w}%
_{\lambda })=0$, i.e., $\bar{w}_{\lambda }$ is a minimizer of (\ref{min1}),  and $u_{\lambda _{m}}\rightarrow \bar{w}_{\lambda }$
strongly in $ W_{per,0}^{1,2}(D_T)$ as $m\rightarrow +\infty$.\QED
\end{proof}
\begin{rem}\label{remapp}
 We note that in Lemma \ref{app}, if $\lambda\in Z$ (as defined in \eqref{eqZ}), then the conclusion of the lemma holds for any sequence $\lambda_m\to\lambda$, as $m\to \infty$. Indeed, since $\lambda\in Z$, we have $P_{\lambda,T}(u_{\lambda})<0$, where $u_\lambda$ is as in the proof of the lemma. Consequently, for large enough $m=1,2,\ldots$, we obtain $P_{\lambda_m,T}(u_{\lambda})\leq 0$. Then the rest of the proof follows exactly as before.
\end{rem}

\section{A Pohozaev type inequality}

In this section, we show Lemma \ref{lem1}. We prove it in the general setting:
\begin{equation*}
	\left\{ \begin{array}{lr}
	&-\Delta u = g(u),~~ (z,x) \in (-T,T)\times \Omega,\\
	&u(-T,x)=u(-T,x), \quad
	u_z(-T,x)=u_z(-T,x),~~x \in \Omega,\\
	&u(z,x)=0,~~x \in \partial \Omega, ~~z \in (-T,T).
	\end{array}
	\right.
\end{equation*}

\begin{lem}
Assume that $\partial \Omega $ is a $C^{2}$-manifold, $N\geq 3$, $g \in C(\mathbb{R})$. Let $u\in C^{2}({D_T})\cap C^{1,\kappa}(\overline{D_T})$ for some $\kappa \in (0,1)$, be a weak solution of \eqref{1}-\eqref{3}. Then the following Pohozaev identity holds:
\begin{equation*}
\frac{1}{2^{\ast }}\int_{D_T} |\nabla_x u|^{2}\,\mathrm{d}x\mathrm{d}z+\frac{1}{2}\int_{D_T} |u_z|^{2}\,\mathrm{d}x\mathrm{d}z+\int_{D_T}G(u)\,\mathrm{d}x\mathrm{d}z=-\frac{1}{2N}\int_{-T}^T\int_{\partial \Omega }\left\vert \frac{%
\partial u}{\partial \nu }\right\vert ^{2}\,(x\cdot \nu (x))\mathrm{d}\sigma
(x).
\end{equation*}
where $G(u)$ is the primitive of $g$ and $\nu$ is the outward normal to $\partial\Omega$.
\end{lem}
\begin{proof} The idea of the proof of this lemma is standard (see \cite{poh, Il_Perid_MatSb, Takac_Ilyasov}), due to this we provide only a sketch here. 
 Denote   $u_{i}=\frac{\partial u}{\partial
x_i}$, for $i=1,2,...,N$. As in \cite{Il_Perid_MatSb, poh}, testing equation \eqref{1} with $u_{i}x_i$ we derive
\begin{eqnarray}
-\int^{T}_{-T}\int_{\Omega}u_{zz}u_{i}x_i\mathrm{d}z\mathrm{d}x -
\int^{T}_{-T}\int_{\Omega}\Delta_xuu_{i}x_i\mathrm{d}z\mathrm{d}x
 = \int^{T}_{-T}\int_{\Omega}g(u)u_{i}x_i\mathrm{d}z\mathrm{d}x.\label{2.3}
\end{eqnarray}
Integrating by parts we obtain
\begin{align}
\Sigma_{i}\int^{T}_{-T}\int_{\Omega}g(u)u_{i}x_i\mathrm{d}z\mathrm{d}x&= \Sigma_{i}
\int^{T}_{-T}\int_{\Omega}\frac{\partial }{\partial x_i}G(u)x_i\mathrm{d}z\mathrm{d}x\nonumber\\
&=-N\int^{T}_{-T}\int_{\Omega}G(u)\mathrm{d}z\mathrm{d}x+
\int^{T}_{-T}\int_{\partial \Omega}G(u)x\cdot\nu dS \label{2.41}
\end{align}
and 
\begin{align}
\int^{T}_{-T}\int_{\Omega}u_{ii}u_{i}x_i\mathrm{d}z\mathrm{d}x 
&=-\frac{1}{2}\int^{T}_{-T}\int_{\Omega}|u_{i}|^2\mathrm{d}z\mathrm{d}x
+\frac{1}{2}\int^{T}_{-T}\int_{\partial \Omega}|u_{i}|^2x_i\nu_idS.
\label{2.6}
\end{align}
For $j\neq i$, performing integrating by parts and after some manipulation, we have
\begin{align}
\int^{T}_{-T}\int_{\Omega}u_{jj}u_{i}x_i\mathrm{d}z\mathrm{d}x
=
\frac{1}{2}\int^{T}_{-T}\int_{\Omega}|u_{i}|^2\mathrm{d}z\mathrm{d}x+
\int^{T}_{-T}\int_{\partial\Omega}u_{j}u_{i}x_i\nu_jdS
-\frac{1}{2}\int^{T}_{-T}\int_{\partial\Omega}|u_{j}|^2x_i\nu_idS.
\label{2.7}
\end{align}
Summing over $j$ and $i\in \{1,2,...,N\}$ the equalities (\ref{2.6}) and (\ref{2.7}) we get
\begin{eqnarray}
 -\Sigma_{i}\int^{T}_{-T}\int_{\Omega}\Delta_xuu_{i}x_i\mathrm{d}z\mathrm{d}x=
-\frac{(N-2)}{2}\int^{T}_{-T}\int_{\Omega}|\nabla_xu|^2\mathrm{d}z\mathrm{d}x 
-\frac{1}{2}\int^{T}_{-T}\int_{\partial\Omega}(\frac{\partial }{\partial\nu}u)^2 (x\cdot\nu)dS,
\label{2.71}
\end{eqnarray}
where we have used the relation $(x\cdot\nu)\frac{\partial}{\partial\nu}u=x\cdot\nabla u$ on $\partial\Omega$. We then integrate by parts first by $z$ and then by $x$. Hence taking into account the periodicity conditions (\ref{2}) and  summing as  above, we deduce that
\begin{align*}
&-\int^{T}_{-T}\int_{\Omega}u_{zz}u_{i}x_i\mathrm{d}z\mathrm{d}x 
=\int^{T}_{-T}\int_{\Omega}u_{z}u_{iz}x_i\mathrm{d}z\mathrm{d}x
-\int_{\Omega}u_{z}(T)u_{i}(T)x_i\mathrm{d}z\mathrm{d}x
+\int_{\Omega}u_{z}(-T)u_{i}(-T)x_i\mathrm{d}z\mathrm{d}x\nonumber\\
&=-\int^{T}_{-T}\int_{\Omega}|u_{z}|^2\mathrm{d}z\mathrm{d}x -
\int^{T}_{-T}\int_{\Omega}u_{z}u_{iz}x_i\mathrm{d}z\mathrm{d}x +\int^{T}_{-T}\int_{\partial
\Omega}|u_{z}|^2x_i\nu_idS.\nonumber
\end{align*}
Hence resolving the last equality and summing as  above we obtain
\begin{eqnarray}
-\Sigma_{i}\int^{T}_{-T}\int_{\Omega}u_{zz}u_{i}x_i\mathrm{d}z\mathrm{d}x=
-\frac{N}{2}\int^{T}_{-T}\int_{\Omega}|u_{z}|^2\mathrm{d}z\mathrm{d}x+
\frac{1}{2}\int^{T}_{-T}\int_{\partial \Omega}|u_{z}|^2 x\cdot\nu dS.
\label{2.81}
\end{eqnarray}
From (\ref{2.41}), (\ref{2.71}),
(\ref{2.81}) and (\ref{2.3}), we obtain
\begin{eqnarray}
&&\hspace*{-0.5cm}\frac{(N-2)}{2N}\int^{T}_{-T}\int_{\Omega}|\nabla_xu|^2\mathrm{d}z\mathrm{d}x+
\frac{1}{2}\int^{T}_{-T}\int_{\Omega}|u_{z}|^2\mathrm{d}z\mathrm{d}x-\int^{T}_{-T}\int_{\Omega}G(u)\mathrm{d}z\mathrm{d}x\nonumber\\
&&=-\frac{1}{N}\int^{T}_{-T}\int_{\partial \Omega}G(u)x\cdot\nu dS
-\frac{1}{2N}\int^{T}_{-T}\int_{\partial
\Omega}(\frac{\partial }{\partial\nu}u)^2 x\cdot\nu dS -\frac{1}{2N}\int^{T}_{-T}\int_{\partial
\Omega}|u_{z}|^2 x\cdot\nu dS. \nonumber
\end{eqnarray}
Since $u(z,\cdot)\equiv 0$ on $\partial \Omega$ for all $z \in [-T,T]$,  we infer that $G(u(z,x))=0$ and $u_z(z,x)=0$ on $\partial \Omega$ for every $z \in [-T,T]$. Hence we conclude the proof of the lemma. \QED
\end{proof}

\end{appendices}

\end{document}